%% file: phi4d.tex
\documentclass[11pt,a4paper]{article}
\usepackage[utf8]{inputenc}
\usepackage[T1]{fontenc}
\usepackage{enumitem}
\setlist{topsep=1pt,parsep=1pt,itemsep=1pt} 

\usepackage{fourier} 

\usepackage{amsmath}
\usepackage{amsfonts}
\usepackage{amssymb}
\usepackage{amsthm}
\usepackage{ upgreek }

\usepackage{amsfonts}

\usepackage{csquotes}

\usepackage{pdflscape}

\usepackage{stmaryrd}

\usepackage[colorlinks=true, linkcolor=teal, linktocpage=true, citecolor=teal]{hyperref} %

\usepackage{dsfont} 
\usepackage{cleveref}

\usepackage{mathrsfs}

\usepackage{tikz-cd}
 
\newtheorem{theorem}{Theorem}[section]
\newtheorem{proposition}[theorem]{Proposition}
\newtheorem{lemma}[theorem]{Lemma}
\newtheorem{corollary}[theorem]{Corollary}

\theoremstyle{definition}
\newtheorem{examplex}[theorem]{Example}
\newenvironment{example}
  {\pushQED{\qed}\examplex}
  {\popQED\endexamplex}
\newtheorem{remarkx}[theorem]{Remark}
\newenvironment{remark}
  {\pushQED{\qed}\remarkx}
  {\popQED\endremarkx}

\numberwithin{equation}{section}

\newcommand{\Wick}[1]{{\mathpunct{:}#1\mathpunct{:}}}

\input{macros.sty}
\input{graphs_gibbs_phi4.sty}

\newcommand{\unit}{\mathbf{1}}

\newcommand{\PiNBPHZ}{\Pi_N^{\mathrm{BPHZ}}}
\DeclareMathOperator{\id}{id}
\newcommand{\DeltaCK}{\Delta_{\mathrm{CK}}}
\newcommand{\DeltaCKred}{\mathring\Delta_{\mathrm{CK}}}
\DeclareMathOperator{\W}{W}
\DeclareMathOperator{\Aut}{Aut}
\newcommand{\ph}{\varphi}

\newcommand{\tM}{\textup{M}}
\newcommand{\tF}{\textup{F}}
\newcommand{\fe}{\mathfrak{e}}
\newcommand{\fl}{\mathfrak{l}}
\newcommand{\fn}{\mathfrak{n}}

\newcommand{\F}{\mathbf{F}}
\newcommand{\Fminus}{\mathbf{F_-}}
\newcommand{\spanF}{\langle\mathbf{F}\rangle}
\newcommand{\spanFminus}{\langle\mathbf{F_-}\rangle}
\newcommand{\spansF}{\langle \mathscr{F}\rangle}
\newcommand{\sFminus}{\mathscr{F}_-}
\newcommand{\spansFminus}{\langle \mathscr{F}_- \rangle}
\newcommand{\M}{\mathbf{M}}
\newcommand{\Mminus}{\mathbf{M_-}}
\newcommand{\spanM}{\langle\mathbf{M}\rangle}
\newcommand{\spanMminus}{\langle\mathbf{M_-}\rangle}
\newcommand{\sM}{\mathscr{M}}
\newcommand{\spansM}{\langle \mathscr{M}\rangle}
\newcommand{\sMminus}{\mathscr{M}_-}
\newcommand{\spansMminus}{\langle \mathscr{M}_- \rangle}

\begin{document}


\title{Perturbative renormalisation of the $\Phi^4_{4-\eps}$ model \\
via generalized Wick maps}
\author{Nils Berglund, Tom Klose, and Nikolas Tapia}
\date{July 4, 2025. Revised version, February 19, 2026.}   

\maketitle

\begin{abstract}
  We consider the perturbative renormalisation of the $\Phi^4_d$ model from Euclidean Quan\-tum Field Theory for any, possibly non-integer dimension $d<4$. The so-called BPHZ renormalisation, named after Bogoliubov, Parasiuk, Hepp and Zimmermann, is usually encoded into extraction-contraction operations on Feynman diagrams, which have a complicated combinatorics. We show that the same procedure can be encoded in the much simpler algebra of polynomials in two unknowns $X$ and $Y$, which represent the fourth and second Wick power of the field. In this setting, renormalisation takes the form of a ``Wick map'' which maps monomials into Bell polynomials. The construction makes use of recent results by Bruned and Hou on multiindices, which are algebraic objects of intermediate complexity between Feynman diagrams and polynomials.
\end{abstract}

\leftline{\small 2020 {\it Mathematical Subject Classification.\/}
81T15, 
81T18, 
16T05.
}



\tableofcontents


\input{introduction.tex}

\subsection*{Acknowledgements}

TK thanks the Institut Denis Poisson in Orléans for its generous financial 
support and the warm hospitality during his visit in April 2025.
All authors thank the referees for their careful reading of the manuscript which lead to an improved presentation.

\subsection*{Funding}

TK is supported by a UKRI Horizon Europe Guarantee MSCA Postdoctoral Fellowship 
(UKRI, SPDEQFT, grant reference EP/Y028090/1). Views and opinions ex\-pres\-sed are however 
those of the author(s) only and do not necessarily reflect those of UKRI. In particular, 
UKRI cannot be held responsible for them.
NT acknowledges funding by the Deutsche Forschungsgemeinschaft (DFG, German Research Foundation) - CRC/TRR 388 \lq\lq Rough Analysis,
Stochastic Dynamics and Related Fields\rq\rq -- Project ID 516748464.


\input{results.tex}


\input{wickmap.tex}


\input{multiindices.tex}


\input{commutative.tex}



\bibliographystyle{plain}
\bibliography{BKT}

\vfill 

\bigskip\bigskip\noindent
{\small
Nils Berglund \\
Institut Denis Poisson (IDP) \\ 
Universit\'e d'Orl\'eans, Universit\'e de Tours, CNRS -- UMR 7013 \\
B\^atiment de Math\'ematiques, B.P. 6759\\
45067~Orl\'eans Cedex 2, France \\
{\it E-mail address: }
{\tt nils.berglund@univ-orleans.fr}
}

\bigskip\bigskip\noindent
{\small
Tom Klose \\
Mathematical Institute \\
University of Oxford \\ 
Woodstock Road \\
Oxford, OX2 6GG, United Kingdom \\
{\it E-mail address:} 
{\tt tom.klose@maths.ox.ac.uk}
}

\bigskip\bigskip\noindent
{\small
Nikolas Tapia \\
Weierstrass Institute for Applied Analysis and Stochastics\\
Mohrenstr. 39\\
10117 Berlin, Germany\\
and\\
Institut für Mathematik\\
Humboldt-Universität zu Berlin\\
Rudower Chaussee 25\\
12486 Berlin, Germany\\
{\it E-mail address: }
{\tt tapia@wias-berlin.de}
}

\end{document}

%% file: introduction.tex
\section{Introduction}
\label{sec:intro}

The $\Phi^4$ model is a famous toy model in Euclidean Quantum Field theory, 
featuring a quartic perturbation to a quadratic Hamiltonian. We are interested 
here in the model defined on the $d$-dimensional torus, which is denoted $\Phi^4_d$. 
In dimension $d=1$, the Gibbs measure associated to the Hamiltonian is well-defined. 
In higher dimension, however, a renormalisation procedure is needed to make sense 
of the Gibbs measure. Such a procedure can be implemented by fixing an ultra-violet 
cut-off $N$, and first restricting the measure to functions whose Fourier modes 
have a wave number of size $N$ at most. This truncated Gibbs measure becomes 
undefined as $N$ is sent to infinity, but in some cases one can fix this problem 
by adding suitable $N$-dependent terms to the Hamiltonian, which diverge as $N$ 
goes to infinity, but make the limiting measure well-defined. 

The easiest case occurs in dimension $d=2$. Then the Gibbs measure can be renormalised 
by replacing the quartic term of the Hamiltonian by its fourth 
Wick power, which is defined by a Hermite polynomial with a variance diverging 
like $\log(N)$. The variance is that of the truncated Gaussian free field describing 
the uncoupled system. In dimension $d=3$, the situation is considerably more difficult. 
In addition to the Wick counterterm, which now diverges like $N$, an additional 
counterterm known as mass renormalisation and which diverges like $\log(N)$, has to 
be added to the quadratic part of the Hamiltonian. In dimensions $d \geq 4$, on the other hand, 
it has been shown that the model is trivial~\cite{Aizenmann-Duminil-Copin, Froehlich_1982}, in the sense 
that any viable renormalised version will converge to a Gaussian model as $N\to\infty$.

Numerous approaches have been developed to deal with the complicated combinatorics 
of the $\Phi^4_3$ model. The earliest works by Glimm and Jaffe and by Feldman 
tackled the problem via a detailed combinatorial analysis of Feynman 
diagrams~\cite{Glimm_Jaffe_68,Glimm_Jaffe_73,Feldman74,Glimm_Jaffe_81}, 
entailing very long and technical proofs. The works~\cite{BCGNOPS78,BCGNOPS_80} 
introduced the idea of using a renormalisation group approach, consisting in a 
decomposition of the covariance of the underlying Gaussian reference measure into 
scales similar to Paley--Littlewood blocks, that can be iteratively integrated out. 
This method was further perfected in~\cite{Brydges_Dimock_Hurd_CMP_95}, 
using polymers to control error terms, an approach based on ideas from 
statistical phy\-sics~\cite{Gruber_Kunz_71}. Another approach, 
provided in~\cite{Brydges_Frohlich_Sokal_CPM83,Brydges_Frohlich_Sokal_CMP83_RW}, 
yields bounds on correlation functions (also known as $n$-point functions), by using the Gibbs measure 
as a generating function. This involves the derivation of skeleton inequalities, 
which were obtained up to third order in~\cite{Brydges_Frohlich_Sokal_CPM83}, 
and later extended to all orders in~\cite{Bovier_Felder_CMP84}. 
A relatively compact derivation of bounds on the partition function based on the Bou\'e--Dupuis 
formula was recently obtained in~\cite{Barashkov_Gubinelli_18}. 
The latter work is part of a programme called \lq\lq stochastic quantisation\rq\rq\ 
which has received a lot of interest in the last ten years due to the advent of novel 
solution techniques for singular stochastic 
PDEs~\cite{Hairer1, GubinelliImkellerPerkowski, kupiainen_renorm_group_spde, Duch_2025}. %
It is far beyond the scope of this introduction to review those recent developments, 
so we only list a few representative works pertaining to the~$\Phi^4_d$ model: 
\cite{mourrat_weber_infinity, hairer_matetski_18_discretisation, 
catellier_chouk_paracontrolled_phi43, gubinelli_hofmanova_elliptic, gubinelli_hofmanova, Shen_Zhu_Zhu_23}. 

The $\Phi^4_d$ model can be extended to non-integer values of the dimension $d$, 
by modifying the power with which the Green function diverges near the origin; 
the corresponding SPDE was introduced in~\cite{BMS2003}.
This makes it possible to study all models for $d\in[3,4)$, and examine their 
behaviour as $d$ approaches the critical value $4$. As $d$ increases, the combinatorics 
becomes increasingly difficult, as more and more counterterms have to be added to the 
Hamiltonian. 
From the viewpoint of stochastic quantisation, this procedure has been understood in a 
non-perturbative way 
in~\cite{Chandra_Moinat_Weber_23, Esquivel_Weber_24, Duch_Gubinelli_Rinaldi_23, Broux_Otto_Steele_25}.

In this work, we expand on an idea introduced in~\cite{Berglund_Klose_22}, 
based on BPHZ renormalisation and inspired by algebraic methods.
BPHZ renormalisation, named after Bogoliubov, Parasiuk, Hepp and 
Zimmermann~\cite{Bogoliubow_Parasiuk,Hepp66,Zimmermann69}, allows to systematically 
analyse the divergence of Feynman diagrams obtained by a perturbative expansion 
of expectations of observables under the Gibbs measure. As realised by Connes and 
Kreimer~\cite{Connes_Kreimer_2000a,Connes_Kreimer_2001}, there is a Hopf-algebraic
structure underlying Feynman diagrams, which encodes the extraction-contraction 
operations occurring in BPHZ renormalisation. See~\cite{Hairer_BPHZ} for a modern 
exposition of this procedure. 

The main new idea in~\cite{Berglund_Klose_22} is that the complicated algebra of 
extraction-contraction operations on Feynman diagrams can be encoded in a much simpler 
way by operations on polynomials in only two variables $X$ and $Y$, representing, 
respectively, the quartic and quadratic parts of the Hamiltonian. This procedure is inspired 
by ideas in~\cite{EFPTZ18} on deformation of copro\-ducts. Our main result, 
\Cref{thm:main}, states that one can indeed encode BPHZ renormalisation on 
the level of polynomials, by using a rather simple \lq\lq generalised Wick map\rq\rq, 
which consists in replacing powers of $X$ by multiples of $Y$, with coefficients 
given by mass renormalisation counterterms. \Cref{cor:BPHZ} then provides 
a simple expression for an asymptotic expansion of the model's partition function. 

The remainder of this paper is organised as follows. \Cref{sec:results} contains 
a definition of the $\Phi^4_d$ model for general, not necessarily integer dimensions $d$, 
introduces perturbative BPHZ renormalisation, and states the main results of this work. In 
\Cref{sec:wick}, we present the construction of the \lq\lq Wick map\rq\rq, and 
discuss its connections to cumulant expansions and Bell polynomials. \Cref{sec:multiindices}
contains the definition of multi-indices in the setting 
of~\cite{bruned2025renormalisingfeynmandiagramsmultiindices}, and an explicit expression 
for the coproduct of the multi-indices representing powers of the quartic interaction term. 
Finally, \Cref{sec:comute} contains the proof of the main result, which amounts 
to showing that a certain diagram between polynomials and multi-indices is commutative.

%% file: results.tex

\section{Main result}
\label{sec:results} 


\subsection[The $\Phi^4_d$ model from Euclidean Quantum Field Theory]
{The $\boldsymbol{\Phi^4_d}$ model from Euclidean Quantum Field Theory}
\label{ssec:phi4d} 

We start by defining the $\Phi^4_d$ measure for integer values of the dimension $d\geqs1$. 
Let $\T^d$ denote the $d$-dimensional torus. For a function $\phi:\T^d\to\R$ and parameters $m>0, \alpha\geqs0$, define the energy 
\begin{equation}
\label{eq:phi4_energy} 
 \mathscr{H}_{d,\alpha}(\phi) = \int_{\T^d} 
 \biggbrak{\frac12 \norm{\nabla\phi(x)}^2 + \frac{m^2}{2} \phi(x)^2 + \alpha \phi(x)^4} \6x\;.
\end{equation} 
The $\Phi^4_d$ measure is the Gibbs measure associated with $\mathscr{H}_{d,\alpha}$, defined formally as 
\begin{equation*}
 \mu_{d,\alpha}(\6\phi) \sim \frac{1}{\sZ_{d,\alpha}} \e^{-\mathscr{H}_{d,\alpha}(\phi)} \6\phi\;,
\end{equation*} 
where the \emph{partition function} $\sZ_{d,\alpha}$ is the normalisation that renders $\mu_{d,\alpha}$ 
a probability measure. 

As such, this definition does not make sense, since there is no Lebesgue measure 
$\6\phi$ on an infinite-dimensional function space. However, in the case $\alpha = 0$, one can 
define $\mu_{d,0}$ as the Gaussian measure with covariance $-\Delta + m^2$, also called (massive) 
\emph{Gaussian 
free field} with this covariance. Taking this Gaussian free field as reference measure, instead of 
Lebesgue measure, one can define the expectation under $\mu_{d,\alpha}$ of a test function $F$ as 
\begin{equation*}
 \expecin{\mu_{d,\alpha}}{F} = \frac{\sZ_{d,0}}{\sZ_{d,\alpha}} 
 \biggexpecin{\mu_{d,0}}{F(\phi) \exp\biggset{-\alpha\displaystyle\int_{\T^d}\phi(x)^4\6x}}\;.
\end{equation*} 
In particular, setting $F = 1$, we obtain 
\begin{equation}
\label{eq:Zratio} 
 \frac{\sZ_{d,\alpha}}{\sZ_{d,0}} = 
 \biggexpecin{\mu_{d,0}}{\exp\biggset{-\alpha\displaystyle\int_{\T^d}\phi(x)^4\6x}}\;.
\end{equation} 
One can show that in dimension $d=1$, the ratio of partition functions is indeed well-defined. 
However, in dimensions $d\geqs2$, this is no longer the case, since the Gaussian free field has 
almost surely infinite variance. Therefore, a renormalisation procedure becomes necessary to 
have a chance of defining the $\Phi^4_d$ measure. 

Given an integer $N\geqs0$, called \emph{ultra-violet cutoff}, denote by $\cK_N$ the space of functions $\phi:\T^d\to\R$ spanned by Fourier basis functions $e_k$ with $\abs{k} = \sum_{i=1}^d \abs{k_i} \leqs N$. 
When restricted to functions in $\cK_N$, the ratio~\eqref{eq:Zratio} is well-defined. However, for $d\geqs2$ 
the ratio does not admit a well-defined limit as $N\to\infty$. 

In dimension $d = 2$, the solution consists in replacing $\phi(x)^4$ by its fourth Wick power, 
defined as 
\begin{equation*}
 \Wick{\phi^4(x)} = H_4(\phi(x),C_N)\;,
\end{equation*} 
where $C_N$ denotes the variance of the Gaussian free field, which diverges like $\log(N)$, 
and $H_4$ is the fourth Hermite polynomial with variance $C_N$, given by 
\begin{equation*}
 H_4(\phi,C_N) = \phi^4 - 6 C_N \phi^2 + 3 C_N^2\;.
\end{equation*} 
Note that this amounts to modifying the mass $m$ in the energy~\eqref{eq:phi4_energy} by a quantity 
depending on the cutoff $N$, and to adding an $N$-dependent constant to the energy. The extra terms 
in the ener\-gy are therefore called \emph{mass renormalisation} and \emph{energy renormalisation} 
counter\-terms. One can then show that with these counterterms included in the definition of 
$\sH_{2,\alpha}$, the ratio~\eqref{eq:Zratio} does converge to a finite limit as $N\to\infty$. 

In dimension $d=3$, Wick renormalisation is no longer sufficient, and additional mass and energy 
renormalisation terms have to be added to the energy. 
In dimension $d=4$, on the other hand, is is known that no renormalisation procedure of the above 
type will lead to a meaningful result. More precisely, the $\Phi^4_4$ model is trivial in the 
sense that the associated renormalised Gibbs measure is Gaussian~\cite{Aizenmann-Duminil-Copin}, 
as is~$\Phi^4_d$ when~$d > 4$~\cite{Froehlich_1982}.

This raises the question whether the limit $d\to 4_-$ can be better understood, by allowing for 
non-integer dimensions. This is known as the $\Phi^4_{4-\eps}$ model. Non-integer dimensions 
can be interpreted, for $3 < d < 4$, by keeping the three-dimensional torus $\T^3$ as domain 
of integration, but changing the behaviour of the \emph{Green function}. The Green function associated 
with the covariance operator $-\Delta+m^2$ on $\T^d$ is defined as 
\begin{equation}
\label{eq:Green_fct} 
 G_d(x,y) = \bigexpecin{\mu_{d,0}}{\phi(x) - \phi(y)} 
 = \sum_{k\in\Z^d} e_k(x) \bigbrak{-\Delta+m^2}^{-1} e_k(y)\;.
\end{equation} 
Note that translation invariance of the model implies that $G_d(x,y) = G_d(x-y,0)$. 
In dimension $d = 3$, the Green function is known to be defined everywhere except on 
the diagonal $\set{x=y}$, and that it diverges like $\norm{x-y}^{-1}$ when approaching the 
diagonal. For $3 < d < 4$, a natural choice for the Green function is thus 
\begin{equation*}
 G_d(x,y) \sim \frac{1}{\norm{x-y}^{d-2}}\;.
\end{equation*} 
The general form of the renormalised energy for ultraviolet cutoff $N$ is expected to be 
\begin{equation}
\label{eq:phi4_energy_renorm} 
 \sH_{d,\alpha,N}(\phi) = \int_{\T^d} 
 \biggbrak{\frac12 \norm{\nabla\phi(x)}^2 + \frac{m^2}{2} \phi(x)^2 + \alpha \Wick{\phi(x)^4}
 + \beta \Wick{\phi(x)^2} + \gamma}  \6x\;,
\end{equation} 
for suitable mass renormalisation $\beta = \beta(d,\alpha,N)$ and energy renormalisation counterterms $\gamma = \gamma(d,\alpha,N)$. Introducing the notations 
\begin{equation*}
 X = \int_{\T^d} \Wick{\phi(x)^4}\6x\;, \qquad 
 Y = \int_{\T^d} \Wick{\phi(x)^2}\6x\;,  
\end{equation*} 
the ratio~\eqref{eq:Zratio} of partition functions can be written as 
\begin{equation}
\label{eq:Zratio_XY} 
 \frac{\sZ_{d,\alpha}}{\sZ_{d,0}} = 
 \bigexpecin{\mu_{d,0}}{\e^{-\alpha X - \beta Y - \gamma}}
 = \e^{-\gamma} \bigexpecin{\mu_{d,0}}{\e^{-\alpha X - \beta Y}}\;.
\end{equation} 


\subsection{Perturbative renormalisation and Feynman diagrams}
\label{ssec:Feynman} 

Perturbative renormalisation consists in expanding the exponential in~\eqref{eq:Zratio_XY}, 
yielding the formal series 
\begin{equation}
\label{eq:Zratio_expansion} 
 \frac{\sZ_{d,\alpha}}{\sZ_{d,0}} \asymp \e^{-\gamma} 
 \sum_{n\geqs 1} \frac{(-1)^n}{n!} \bigexpecin{\mu_{d,0}}{(\alpha X + \beta Y)^n}\;.
\end{equation} 
We emphasize that this series is known \emph{not} to converge~\cite{Jaffe_1965}. 
Instead, it is an asymptotic series 
of Gevrey-1 type, which is Borel summable (at least for 
$d\leqs 3$, cf.~\cite{Eckmann_Magnen_Seneor_75,Magnen_Seneor_77}). 

By the Isserlis--Wick theorem, expectations of monomials in $X$ and $Y$ can be written in 
terms of integrals of products of Green functions. For instance,  
\begin{align*}
 \bigexpecin{\mu_{d,0}}{X^2} 
 &= \int_{\T^d} \int_{\T^d} \bigexpecin{\mu_{d,0}}{\Wick{\phi(x)^4} \, \Wick{\phi(y)^4}} \6x\6y \\
 &= \int_{\T^d} \int_{\T^d} 4! \, \bigexpecin{\mu_{d,0}}{\phi(x)\phi(y)}^4 \6x\6y \\
 &= 4! \int_{\T^d} \int_{\T^d} G_{d,N}(x,y)^4 \6x\6y\;.
\end{align*} 
Here the truncated Green function $G_{d,N}$ is defined as in~\eqref{eq:Green_fct}, 
with the power $-1$ replaced by a suitable power $-s$ --- power-counting arguments imply 
that $s=(7-d)/4$. In addition, the sum should be 
restricted to $k\in\Z^3$ with $\abs{k}\leqs N$. This can be conveniently represented 
by the graphical notation 
\begin{equation*}
 \bigexpecin{\mu_{d,0}}{X^2} 
 = 4! \, \Pi_N \FGIV\;, 
\end{equation*} 
where the vertices of the graph correspond to the two integration variables, the four 
edges correspond to the four copies of the Green function, and $\Pi_N$ denotes the \emph{valuation
map}, which associates to the diagram a real number given by an integral of a product of 
Green functions. This is an example of \emph{Feynman diagram} (in this case, a vacuum diagram, 
since there are no free legs). 

For a general multigraph $\Gamma = (\sV,\sE)$ (multiple edges are allowed), the valuation is defined by 
\begin{equation*}
 \Pi_N(\Gamma) = \int_{(\T^3)^{\abs{\sV}}} \prod_{e\in\sE} G_{d,N}(x_{e_+},x_{e,-}) \6x\;, 
\end{equation*} 
where $x_{e_+}$ and $x_{e_-}$ denote the endpoints of the edge $e$. We further represent 
the Wick powers $X$ and $Y$ by 
\begin{equation}
\label{def:XY} 
 X = \FGfour, \qquad
 Y = \FGtwo\;.
\end{equation} 
With these notations, the expectation of a monomial is given by 
\begin{equation*}
 \bigexpecin{\mu_{d,0}}{X^nY^m}
 = \sum_p \Pi_N(\Gamma_p)\;,
\end{equation*} 
where $p$ runs over all \emph{pairwise matchings} of the legs of $X$ and $Y$. 
Since we are taking the expectation, these pairwise matchings need to be \emph{perfect}, i.e. they \emph{cannot} leave any free legs, corresponding to the above-mentioned case of vacuum diagrams.

\smallskip
\noindent
\textbf{Convention:}
From here onwards, we will simple write \enquote{(Feynman) diagrams} to mean \enquote{vacuum Feynman diagrams} as those are all we need.

Instead of expanding the exponential directly as in~\eqref{eq:Zratio_expansion}, it turns out 
to be advantageous to expand its logarithm, that is, to use a \emph{cumulant expansion}. The 
\emph{linked-cluster theorem}~\cite{Brouder09, Rivasseau_09, Salmhofer_Renormalization} states that 
\begin{equation}
\label{eq:cumulant_expansion} 
 \log \frac{\sZ_{d,\alpha}}{\sZ_{d,0}} 
 = \log \bigexpecin{\mu_{d,0}}{\e^{-\alpha X - \beta Y - \gamma}}
 \asymp 
 -\gamma + \sum_{n\geqs0} \frac{(-1)^n}{n!} \Pi_N \sP\bigpar{(\alpha X + \beta Y)^n}\;,
\end{equation} 
where $\sP$ denotes the sum over \emph{connected} pairwise matchings, i.e. connected vacuum diagrams.
In some more detail, the map~$\sP$ encodes the following operation: Probabilistically, $X$ and~$Y$ live in the $4$-th respectively the $2$-nd homogeneous Wiener chaos and their (joint) moments can be computed via the so-called diagram formula, see \cite[Section 7]{Peccati_Taqqu_book}. This gives rise to a sum over Feynman diagrams which, as is part of~$\sP$, are then projected onto \emph{connected} diagrams. We refer to \cite[Remark~3.3]{Berglund_Klose_22} for details.


\subsection{BPHZ renormalisation}
\label{ssec:BPHZ} 

BPHZ renormalisation, a procedure named after Bogoliubov, Parasiuk, Hepp and 
Zimmermann~\cite{Bogoliubow_Parasiuk, Hepp66, Zimmermann69}, allows to 
determine expressions for the counterterms $\beta$ and $\gamma$ ensuring that the 
terms obtained by expanding~\eqref{eq:cumulant_expansion} are uniformly bounded in 
the cut-off $N$. 

The first step is to associate to a diagram $\Gamma = (\sV,\sE)$ the \emph{degree} 
\begin{equation} \label{e:deg_graphs}
 \deg(\Gamma) = d(|\sV|-1) - (d-2)|\sE|\;.
\end{equation} 
Diagrams of non-positive degree are called \emph{divergent}. 
Simple examples suggest that $\Pi_N(\Gamma)$ diverges like $N^{-\deg(\Gamma)}$ 
if $\deg(\Gamma) < 0$, and like $\log(N)$ if $\deg(\Gamma) = 0$, while 
$\Pi_N(\Gamma)$ is uniformly bounded in $N$ if $\deg(\Gamma) > 0$. This is not true in general, because 
a non-divergent diagram can contain subdiagrams that are divergent, making the 
valuation diverge. However, a deep result of the theory states that there exists 
a modification of $\Gamma$ which behaves like its degree.

\begin{theorem}[BPHZ renormalisation]
\label{thm:BPHZ} 
There exists a linear map $\sA$, acting on Feynman diagrams, such that 
\begin{equation*}
 \Pi_N(\sA(\Gamma)) \asymp 
 \begin{cases}
  N^{-\deg(\Gamma)} & \text{if $\deg(\Gamma)<0$\;,} \\
  \log(N)^\zeta  & \text{if $\deg(\Gamma)=0$\;,}
 \end{cases}
\end{equation*}
for a finite integer $\zeta$, while $\Pi_N(\sA(\Gamma))$ is bounded 
uniformly in $N$ if $\deg(\Gamma) > 0$. 
\end{theorem}

%
For a modern exposition, see~\cite{Hairer_BPHZ}, as well as~\cite{BB19}.
The linear map $\sA$ is in fact the antipode of the Connes--Kreimer 
Hopf algebra on Feynman diagrams. 
To define this Hopf algebra, we introduce the following spaces:
\begin{itemize}
\item $\F$ is a collection of \emph{connected} multigraphs, and $\spanF$ denotes its linear span; in our case, $\F$ contains all graphs whose vertices have arity $2$, $3$ or $4$ 
(that is, $2$, $3$ or $4$ edges meet at each vertex); 
\item $\Fminus \subset \F$ denotes the divergent multigraphs in~$\F$, and $\spanFminus$ is its linear span;  
\item $\sF$ is the algebra generated by~$\F$ w.r.t. the disjoint union product~$\cdot$ \thinspace, and~$\spansF$ is its linear span -- note that both~$\sF$ and~$\spansF$ also contain \emph{non-connected} Feynman diagrams;
\item $\sFminus \subset \sF$ is the subalgebra of~$\sF$ generated by~$\Fminus$; in particular, for any~$\Gamma \in \sFminus$, all connected components are divergent. 
We also let~$\spansFminus$ denote the span of~$\sFminus$. 
\end{itemize}
Recall our earlier convention that all graphs involved in these definitions are vacuum diagrams, i.e. do not have free legs.
The neutral element for multiplication 
is the empty graph (resp. empty forest) which we denote by $\unit$. The \emph{Connes--Kreimer extraction-contraction 
co\-product} $\DeltaCK: \spanF \to \spansFminus \otimes \spanF$ is defined by 
\begin{equation}
 \label{eq:CK_coproduct} 
 \DeltaCK(\Gamma) = \Gamma\otimes\unit + \unit\otimes\Gamma 
+ \sum_{\substack{\unit\neq\overline\Gamma\subsetneq\Gamma, \overline\Gamma \in \Fminus}} 
\overline\Gamma\otimes(\Gamma/\overline\Gamma)\;, 
\end{equation} 
where the sum ranges over all \emph{divergent} subgraphs $\overline\Gamma$, and 
$\Gamma/\overline\Gamma$ denotes the graph obtained by replacing $\overline\Gamma$
by a single vertex. The subgraphs have to be \emph{full}, in the sense that if an edge
$e$ belongs to $\overline\Gamma$, all edges connecting the same vertices also belong 
to $\overline\Gamma$. Note that the valuation $\Pi_N$ is multiplicative, meaning 
that $\Pi_N(\Gamma_1\cdot\Gamma_2) = \Pi_N(\Gamma_1)\Pi_N(\Gamma_2)$ for all 
$\Gamma_1,\Gamma_2\in\sF$. 

\begin{remark}
\label{rem:Taylor} 
We point out that if some divergent subgraphs have a degree smaller than $-1$, 
the expression~\eqref{eq:CK_coproduct} has to be modified, by adding decorated graphs on the 
right-hand side, as we will explain in \Cref{ssec:Taylor} below.  
\end{remark}

We endow $\spanF$ with two more linear maps. A \emph{counit} $\unit^\star: \spanF \to\R$, given by projection on the unit $\unit$, and an \emph{antipode} $\sA:\spanF\to\spansF$, defined inductively by $\sA(\unit)=\unit$ and 
\begin{align}
 \sA(\Gamma) &= -\Gamma - \sum_{\substack{\unit\neq\overline\Gamma\subsetneq\Gamma, \overline\Gamma \in \Fminus}} 
\sA(\overline\Gamma)\cdot(\Gamma/\overline\Gamma)\\
&= - \Gamma - m(\sA\otimes\id) \DeltaCKred(\Gamma)\;.
\label{eq:antipode} 
\end{align} 
Here $\DeltaCKred = \DeltaCK - \Gamma\otimes\unit - \unit\otimes\Gamma$ 
denotes the \emph{reduced coproduct}, and the map $m:\spansF\otimes \spansF\to\spansF$ denotes multiplication, defined 
by $m(\Gamma_1\otimes\Gamma_2) = \Gamma_1 \cdot \Gamma_2$. 
The antipode~$\sA$ as well as the maps~$\DeltaCK$ and 
$\unit^\star$ can be multiplicatively extended to the algebra~$\sF$.
The space $(\sF,\cdot,\DeltaCK,\unit,\unit^\star,\sA)$ constructed in this way is a Hopf 
algebra, called \emph{Connes--Kreimer extraction-contraction Hopf algebra}. 

To define BPHZ renormalisation, we first introduce the \emph{twisted antipode},  
defined as 
\begin{equation*}
 \tilde \sA(\Gamma) = \sA(\Gamma) 1_{\deg\Gamma\leqs0}\;.
\end{equation*} 
Note that if $\deg(\Gamma)\leqs 0$, then one has 
\begin{equation}
\label{eq:twisted_antipode} 
 \tilde \sA(\Gamma) = - \Gamma - m(\tilde\sA\otimes\id) \DeltaCKred(\Gamma)\;,
\end{equation} 
because $\DeltaCKred$ produces only divergent terms on the left of the tensor product. 

\begin{remark}
The general definition of the twisted antipode given in \cite[(2.28)]{Hairer_BPHZ}
is rather that it should satisfy $m(\tilde\sA\otimes\id)\DeltaCK(\Gamma) = 0$ whenever $\deg\Gamma \leqs 0$, 
where $m$ denotes the multiplication map, but this is equivalent -- cf.\ 
\cite[(4.3)]{bruned2025renormalisingfeynmandiagramsmultiindices}, which uses the reduced coproduct. 
\end{remark}

A \emph{character} on $\spansF$ is a linear map $g:\spansF\to\R$ which is multiplicative 
in the sense that we have $g(\Gamma_1\cdot\Gamma_2) = g(\Gamma_1)g(\Gamma_2)$ for all $\Gamma_1, \Gamma_2\in\spansF$. 
With any character $g$, one can associate a linear map 
$M^g$ defined by 
\begin{equation*}
 M^g(\Gamma) = (g\otimes\id)\DeltaCK(\Gamma)\;,
\end{equation*} 
and the set of these maps is known to form a group. 
The \emph{BPHZ character} is the linear map $g^{\text{BPHZ}}: \spansF\to\R$ given by 
\begin{equation*}
 g^{\text{BPHZ}}(\Gamma) = \Pi_N \tilde\sA(\Gamma)\;.
\end{equation*} 
The fact that $g^{\text{BPHZ}}$ is indeed a character follows from multiplicativity 
of $\sA$ and $\Pi_N$. The map $M^{g^{\text{BPHZ}}}$ is called \emph{BPHZ renormalisation map}. 
It defines a \emph{renormalised valuation} given by 
\begin{equation}
\label{eq:PiNBPHZ} 
 \PiNBPHZ(\Gamma) 
 = \Pi_N M^{g^{\text{BPHZ}}}(\Gamma)
 = (g^{\text{BPHZ}}\otimes\Pi_N)\DeltaCK(\Gamma)
 = (\Pi_N\tilde \sA \otimes \Pi_N)\DeltaCK(\Gamma)\;.
\end{equation} 
The interest of this construction is the following result. 

\begin{lemma}
\label{lem:BPHZ} 
The BPHZ renormalised valuation satisfies 
\begin{equation*}
 \PiNBPHZ(\Gamma) = 
 \begin{cases}
  0 & \text{if\/ $\deg\Gamma\leqs0$\;,}\\
  -\Pi_N\sA(\Gamma) & \text{if\/ $\deg\Gamma > 0$\;.}
 \end{cases}
\end{equation*} 
\end{lemma}
\begin{proof}
In the case $\deg\Gamma\leqs0$, using~\eqref{eq:twisted_antipode} we get 
\begin{align*}
 \PiNBPHZ(\Gamma) 
 ={}& (\Pi_N\otimes\Pi_N)(\tilde\sA\otimes\id)
 \bigbrak{\Gamma\otimes\unit + \unit\otimes\Gamma + \DeltaCKred(\Gamma)}\\
 ={}& (\Pi_N\otimes\Pi_N)\bigbrak{\tilde\sA(\Gamma)\otimes\unit + \unit\otimes\Gamma 
 + (\tilde\sA\otimes\id)\DeltaCKred(\Gamma)} \\
 ={}& (\Pi_N\otimes\Pi_N)
 \bigbrak{-\Gamma\otimes\unit 
 - m(\tilde\sA\otimes\id) \DeltaCKred(\Gamma)\otimes\unit \\ 
 &\phantom{(\Pi_N\otimes\Pi_N)[} 
 {}+ \unit\otimes\Gamma + (\tilde\sA\otimes\id)\DeltaCKred(\Gamma)}\;,
\end{align*}
which vanishes by multiplicativity of $\Pi_N$. 
In the case $\deg\Gamma > 0$, using $\tilde\sA\Gamma = 0$ in the second line of the 
above computation, we obtain 
\begin{align*}
 \PiNBPHZ(\Gamma)
 &= (\Pi_N\otimes\Pi_N)\bigbrak{\unit\otimes\Gamma 
 + (\tilde\sA\otimes\id)\DeltaCKred(\Gamma)} \\
 &= \Pi_N(\Gamma) + (\Pi_N\tilde\sA \otimes \Pi_N)\DeltaCKred(\Gamma) \\
 &= \Pi_N(\Gamma) + \Pi_N m(\tilde\sA \otimes \id)\DeltaCKred(\Gamma)
\end{align*}
again by multiplicativity of $\Pi_N$. This is equal to $-\Pi_N\sA(\Gamma)$
by~\eqref{eq:antipode}. 
\end{proof}

If follows from \Cref{thm:BPHZ} that the renormalized valuation $\PiNBPHZ$
is bounded uniformly in the cut-off $N$ for any $\Gamma\in\spansF$. 


\subsection{Main result}
\label{ssec:main_result} 

To formulate our main result, we introduce two sequences of critical dimensions,
at which new energy or mass renormalisation terms appear. 
Note that any graph in $\sP(X^n)$ has $n$ vertices and $4n$ half-edges, which become $2n$
edges after pairing. Therefore, 
\begin{equation*}
 \deg \sP(X^n) = 4n - (n+1) d\;.
\end{equation*} 
This implies that 
\begin{equation*}
 \deg \sP(X^n) \leqs 0
 \qquad \Leftrightarrow \qquad 
 d \geqs d^*_{\textup{e}}(n) := \frac{4n}{n+1}
 = 4 - \frac{4}{n+1}\;.
\end{equation*} 
These thresholds control the appearance of new energy renormalisation counterterms. 
New mass renormalisation terms appear at the values 
\begin{equation*}
 d^*_{\textup{m}}(n) 
 = d^*_{\textup{e}}(2n-1)
 = 4 - \frac{2}{n} 
\end{equation*} 
of $d$. The two sequences are increasing and accumulate at $d=4$. The first values are 
\begin{equation*}
 \bigpar{d^*_{\textup{e}}(n)}_{n\geqs1} 
 = \biggpar{2, \frac83, 3, \frac{16}{5}, \frac{10}{3}, \frac{24}{7}, \frac72, \frac{32}9, \dots}
 \quad \text{and} \quad 
 \bigpar{d^*_{\textup{m}}(n)}_{n\geqs1} = \biggpar{2, 3, \frac{10}{3}, \frac72, \dots}\;.
\end{equation*} 
The inverse thresholds, expressing $n$ in terms of $d$, are given by 
\begin{equation*}
 n^*_{\textup{e}}(d) = \biggl\lfloor \frac{d}{4-d} \biggr\rfloor
 \qquad \text{and} \qquad 
 n^*_{\textup{m}}(d) = \biggl\lfloor \frac{2}{4-d} \biggr\rfloor\;.
\end{equation*} 
Our main result is the following.

\begin{theorem}[Main result]
\label{thm:main} 
For any dimension $d < 4$, there exists a linear map $\W:\R[X]\to\R[X,Y]$, called \emph{Wick map}, such that the following diagram commutes: 
\begin{equation}  
\label{comdiag:C[X]G}
\begin{tikzcd}[column sep=large, row sep=large]
\R[X]
\arrow[r, "\sP"] 
\arrow[d, "\W"']
& \spanF
\arrow[d, "(\Pi_N\tilde{\sA}\otimes\id)\DeltaCK + \Theta_\tF"] 
\\
\displaystyle
\R[X,Y]
\arrow[r, "\sP"] 
& \spanF
\end{tikzcd}
\end{equation}
The Wick map $\W$ satisfies 
\begin{equation*}
 \W(\e^{-\alpha X}) = \e^{-\alpha X - \beta Y}\;,
\end{equation*} 
where the mass counterterm $\beta = \beta(d,\alpha,N)$ is given by 
\begin{equation} \label{e:beta}
 \beta(d,\alpha,N) = \sum_{n=2}^{n^*_{\textup{m}}(d)} \frac{(-\alpha)^n}{n!}\sigma_n(N) \;.
\end{equation}
Here the $\sigma_n(N)$ can be expressed in terms of divergent Feynman diagrams, and diverge like 
$\sigma_n(N) \sim N^{2 - (4-d)n}$, see \Cref{rmk:sigma_n} below. 
In addition, for any $n\geqs2$, 
\begin{equation*}
 \W(X^n) = B_n(X,-\sigma_2(N) Y, \dots, -\sigma_n(N) Y)\;,
\end{equation*} 
where $B_n$ is the $n$th complete Bell polynomial. Finally, the map $\Theta_\tF$
is associated with energy renormalisation, in that the energy 
counterterm $\gamma = \gamma(d,\alpha,N)$ is given by 
\begin{equation*}
\gamma
 := (\Pi_N\Theta_\tF\circ\sP)(\e^{-\alpha X})
 = - (\Pi_N\tilde\sA \circ \sP)(\e^{-\alpha X})\;.
\end{equation*}
\end{theorem}

Together with the definition~\eqref{eq:PiNBPHZ} of the BPHZ valuation, this implies 
that the following diagram commutes: 
\begin{equation}  
\label{comdiag:BPHZ}
\begin{tikzcd}[column sep=large, row sep=large]
\e^{-\alpha X}
\arrow[r, "\sP", mapsto] 
\arrow[d, "\W"', mapsto]
& \sP(\e^{-\alpha X})
\arrow[d, "(\Pi_N\tilde{\sA}\otimes\id)\DeltaCK + \Theta_\tF", mapsto] 
\arrow[rrd, bend left, "\PiNBPHZ + \Pi_N\Theta_\tF"]
\\
\e^{-\alpha X - \beta Y}
\arrow[r, "\sP", mapsto] 
& \sP(\e^{-\alpha X - \beta Y})
\arrow[rr, "\Pi_N", mapsto] 
&& \R
\end{tikzcd}
\end{equation}
As a consequence, we have 
\begin{equation}
\label{eq:logZ_sum} 
 \log\frac{\sZ_{d,\alpha}}{\sZ_{d,0}}
 = \Pi_N\sP(\e^{-\alpha X-\beta Y}) -\gamma
 = \PiNBPHZ\sP(\e^{-\alpha X})\;.
\end{equation} 
\Cref{lem:BPHZ} shows that $\PiNBPHZ\sP(X^n) = 0$ for $n\leqs n^*_{\textup{e}}(d)$, 
while it is bounded for $n> n^*_{\textup{e}}(d)$. 
\Cref{thm:BPHZ} then immediately implies the following result. 

\begin{corollary}
\label{cor:BPHZ} 
With the choice~\eqref{eq:gamma} of the energy renormalisation term, one has the asymptotic 
expansion 
\begin{equation*}
 \log\frac{\sZ_{d,\alpha}}{\sZ_{d,0}} 
 \asymp - \sum_{n > n^*_{\textup{e}}(d)}\frac{(-\alpha)^n}{n!} \Pi_N\sA(\sP(X^n))\;.
\end{equation*}
All terms of this expansion are bounded uniformly in the cut-off $N$. 
\end{corollary}

\begin{remark} \label{rmk:gamma}
The energy renormalisation counterterm can be written more explicitly as 
\begin{equation}
\label{eq:gamma} 
\gamma =  \gamma(d,\alpha,N) = 
 - \sum_{n=2}^{n^*_{\textup{e}}(d)}
 \frac{(-\alpha)^n}{n!} \Pi_N\tilde\sA(\sP(X^n))\;,
\end{equation} 
where $\Pi_N\tilde\sA(\sP(X^n))$ diverges like $N^{(n+1)d-4n} = N^{(4-d)(n^*_{\textup{e}}(d)-n)}$, 
with the tacit understanding that~$N^0 = \log N$.
\end{remark}

\begin{remark} \label{rmk:sigma_n}
The counterterms $\sigma_n$ occuring in the mass renormalisation $\beta$
are linear combinations of valuations of divergent Feynman diagrams, the first of which 
are listed in \Cref{tab:subdivergence}. More precisely, they can be written 
\begin{equation*}
 \sigma_n = -\Pi_N\tilde\sA\sP_\tM(\sY_n)\;,
\end{equation*} 
where the $\sY_n$ are defined in~\eqref{eq:def_Yp} below, and the map $\sP_\tM$ is 
defined in~\eqref{eq:PM}. They can also be computed by the relation 
\begin{equation*}
 \sigma_n = -\Pi_\tM\tilde\sA_\tM(\sY_n)\;,
\end{equation*} 
where the maps $\Pi_\tM$ and $\tilde\sA_\tM$ are defined 
in~\cite[Corollary~4.5 and (4.7)]{bruned2025renormalisingfeynmandiagramsmultiindices} 
-- the twisted antipode is denoted $\sA_\tM$ in that work. 
We note that the general algebraic structure of the mass counterterms~$\beta$ given 
in~\eqref{e:beta} also appears in~\cite[Theorem~3.1]{bruned2025renormalisingfeynmandiagramsmultiindices}.
\end{remark}

\subsection{Structure of the proof}

We start, in \Cref{sec:wick}, by explaining the construction of the Wick map $\W$, 
and its (well known) relation to Bell polynomials. To prove the commutativity of the 
diagram~\eqref{comdiag:C[X]G}, we will take advantage of recent results 
in~\cite{bruned2025renormalisingfeynmandiagramsmultiindices}, which show that instead 
of working with Feynman diagrams, one can work with somewhat simpler algebraic objects 
called~\emph{multi-indices}. We introduce these objects in \Cref{sec:multiindices}, 
where we also compute an explicit expression for the coproducts of the relevant 
Feynman diagrams, translated to the multi-index language. Finally, \Cref{sec:comute} 
contains the proof of the main result.

%% file: wickmap.tex
\section{The Wick map}
\label{sec:wick} 

In this section, we give the construction of the Wick map $\W$, and explain its 
connection to cumulants and Bell polynomials. The Wick map can be cast into a form 
closer to the map $(\Pi_N\sA\otimes\id)\DeltaCK$ by working in the free 
commutative Hopf algebra $H = S(\R[X])$, as explained in \Cref{ssec:free_algebra}. 
In \Cref{ssec:algebra_valued}, we explain how the construction extends 
to algebra-valued moments. 

\subsection{Convolution algebra and power series}

Let $\R[X]$ denote the polynomial Hopf algebra on a single formal variable $X$. 
Its product and coproduct are given by 
\begin{equation*}
 X^n\cdot X^m = X^{n+m}\;, \qquad 
 \Delta X^n = \sum_{k=0}^n \binom{n}{k} X^k \otimes X^{n-k}\;.
\end{equation*} 
The \emph{convolution product} of two linear maps $\ph,\psi\in\sL(\R[X],\R)$ is defined 
by 
\begin{equation*}
 \ph \ast \psi = m_\R (\ph\otimes\psi) \Delta\;,
\end{equation*} 
where $m_\R$ denotes the multiplication map, $m_\R(a\otimes b) = ab$. 
This means that 
\begin{equation*}
 (\ph\ast\psi)(X^n) = \sum_{k=0}^n \binom{n}{k} \ph(X^k)\psi(X^{n-k})\;.
\end{equation*} 
The $k$-fold convolution product is denoted $\ph^{\ast k}$. 

We define two special subsets of $\sL(\R[X],\R)$, given by 
\begin{align*}
 \sL_1 &= \bigsetsuch{\ph\in\sL(\R[X],\R)}{\ph(1) = 1}\;, \\
 \sL_0 &= \bigsetsuch{\ph\in\sL(\R[X],\R)}{\ph(1) = 0}\;. 
\end{align*}
Elements of $\sL_1$ can be inverted, via the Neumann series 
\begin{equation*}
 \ph^{-1} = \sum_{k=0}^\infty(\eps - \ph)^{\ast k}\;,
\end{equation*} 
where $\eps:\R[X]\to\R$ is the counit, given by $\eps(X^n) = \delta_{n0}$. 
One has explicitly 
\begin{equation}
\label{eq:phi_inverse} 
 \ph^{-1}(X^n) = \sum_{k=1}^n (-1)^k \sum_{\substack{n_1,\dots,n_k\geqs1\\n_1+\dots+n_k=n}}
 \frac{n!}{n_1!\dots n_k!} \ph(X^{n_1})\dots\ph(X^{n_k})\;.
\end{equation} 
The exponential map $\exp_\ast:\sL_0\to\sL_1$ and its inverse $\log_\ast:\sL_1\to\sL_0$ are given by 
\begin{equation*}
 \exp_\ast(\ph) = \sum_{k\geqs0} \frac{1}{k!} \ph^{\ast k}\;, \qquad 
 \log_\ast(\ph) = \sum_{k\geqs1} \frac{(-1)^k}{k} (\ph - \eps)^{\ast k}\;.
\end{equation*} 
There is no issue of convergence, since the sums are always finite when evaluated on a basis element.
In fact,
\begin{align}
\label{eq:expstar}
 \exp_\ast(\ph)(X^n) &= 
 \sum_{k=0}^n \frac{1}{k!}
 \sum_{\substack{n_1, \dots, n_k\geqs1\\n_1+\dots+n_k=n}}
 \frac{n!}{n_1!\dots n_k!}\ph(X^{n_1})\dots\ph(X^{n_k})\;, \\
 \log_\ast(\ph)(X^n) &= 
 \sum_{k=1}^n \frac{(-1)^{k+1}}{k}
 \sum_{\substack{n_1, \dots, n_k\geqs1\\n_1+\dots+n_k=n}}
 \frac{n!}{n_1!\dots n_k!}\ph(X^{n_1})\dots\ph(X^{n_k})\;.
\label{eq:logstar} 
\end{align} 
Let $\R[[t]]$ denote the algebra of (formal) power series in the variable $t$ with real coefficients, 
endowed with pointwise multiplication. Then we can define a linear map
\begin{align*}
 \Lambda: \sL(\R[X],\R) &\to \R[[t]] \\
 \ph &\mapsto \sum_{n\geqs0} \ph(X^n) \frac{t^n}{n!}\;.
\end{align*} 

\begin{theorem}
\label{thm:conv} 
$\Lambda$ is an isomorphism between $\sL(\R[X],\R)$ and $\R[[t]]$.
\end{theorem}
\begin{proof}
Let $\ph,\psi\in \sL(\R[X],\R)$. By the Cauchy product formula,
\begin{align*}
 \Lambda(\ph)(t) \Lambda(\psi)(t) 
 &= \biggpar{\sum_{n\geqs0} \ph(X^n) \frac{t^n}{n!}}
 \biggpar{\sum_{n\geqs0} \psi(X^n) \frac{t^n}{n!}} \\
 &= \sum_{n\geqs0} \biggpar{\sum_{k=0}^n \frac{\ph(X^k)}{k!}\frac{\ph(X^{n-k})}{(n-k)!}} t^n \\
 &= \sum_{n\geqs0} (\ph\ast\psi)(X^n) \frac{t^n}{n!} \\
 &= \Lambda(\ph\ast\psi)(t)\;,
\end{align*}
showing that $\Lambda$ is indeed a morphism. Bijectivity of $\Lambda$ is straightforward to check.
\end{proof}

\begin{corollary}
\label{cor:Lambda} 
For any $\ph\in\sL_0$ and $\psi\in\sL_1$, one has the relations 
\begin{align}
 \Lambda(\psi^{-1})(t) &= \bigbrak{\Lambda(\psi)(t)}^{-1}\;, \\
 \Lambda(\exp_\ast\ph)(t) &= \exp(\Lambda(\ph)(t))\;, \\
 \Lambda(\log_\ast\psi)(t) &= \log(\Lambda(\psi)(t))\;.
 \label{eq:Lambda_logstar} 
\end{align}
Recall that $\psi^{-1}$ denotes the map satisfying $\psi * \psi^{-1} = \eps$.
\end{corollary}
\begin{proof}
We prove the second relation. Setting $\psi = \exp_\ast\ph$, we have
\begin{equation*}
 \Lambda(\psi)(t) = \sum_{k\geqs0} \frac{1}{k!} \Lambda(\ph^{\ast k})(t) 
 = \sum_{k\geqs0} \frac{1}{k!} \Lambda(\ph)(t)^k 
 = \exp(\Lambda(\ph)(t))\;.
\end{equation*} 
The other relations are proved in a similar way. 
\end{proof}

\subsection{Moments, cumulants and Wick exponential}

Let $\sX$ be a real-valued random variable having moments of all orders. 
We associate with it the linear map $\mu_\sX:\R[X]\to\R$ given by 
\begin{equation*}
 \mu_\sX(X^n) = \expec{\sX^n}\;.
\end{equation*} 
Note that $\mu_\sX\in\sL_1$, since $\expec{1} = 1$. The associated power series 
\begin{equation}
\label{eq:moment_generating} 
 \Lambda(\mu_\sX)(t) = \sum_{n\geqs0} \frac{t^n}{n!} \expec{\sX^n}
 = \expec{\e^{t\sX}}
\end{equation} 
is the \emph{moment generating function} of $\sX$.
The \emph{cumulant generating function} of $\sX$ is defined as 
\begin{equation*}
 K_\sX(t) = \log\expec{\e^{t\sX}}\;.
\end{equation*} 
\Cref{cor:Lambda} implies 
\begin{equation*}
 K_\sX(t) = \Lambda(\log_\ast \mu_\sX)(t) 
 = \Lambda(\kappa_\sX)(t)\;,
\end{equation*} 
where
\begin{equation*}
 \kappa_\sX = \log_\ast \mu_\sX\;.
\end{equation*} 
This is nothing but the classical moment-cumulant relation. In particular, \eqref{eq:logstar}
implies
\begin{equation*}
 \kappa_\sX(X^n) 
 = \sum_{k=1}^n \frac{(-1)^{k+1}}{k}
  \sum_{\substack{n_1,\dots,n_k\geqs1\\n_1+\dots+n_k=n}}
 \frac{n!}{n_1!\dots n_k!} \mu_\sX(X^{n_1})\dots\mu_\sX(X^{n_k})\;.
\end{equation*} 
We now define the \emph{Wick exponential} associated to $\sX$ by 
\begin{equation}
\label{eq:Wick_expo} 
 \W = (\mu_\sX^{-1} \otimes \id) \Delta 
 = (\exp_\ast(-\kappa_\sX) \otimes \id) \Delta\;. 
\end{equation} 
One easily checks that $\W$ is an element of $\sL_1$, while~\eqref{eq:phi_inverse} and~\eqref{eq:expstar}
imply
\begin{equation}
\label{eq:W(Xn)} 
 \W(X^n) 
 = \sum_{k=0}^n \sum_{j=1}^k \frac{(-1)^j}{j!}
 \sum_{\substack{k_1,\dots,k_j\geqs1\\k_1+\dots+k_j=k}}
 \frac{n!}{(n-k)!k_1!\dots k_j!} \kappa_\sX(X^{k_1})\dots\kappa_\sX(X^{k_j})X^{n-k}\;.
\end{equation} 
The following result shows that $\W(t,X) := \Lambda(\W)(t)$ can 
be interpreted as a generating function. 

\begin{proposition}
\label{prop:Wick_map} 
One has the relation
\begin{equation*}
 \W(t,X) 
 = \frac{\e^{tX}}{\expec{\e^{t\sX}}}
 = \e^{tX - K_\sX(t)}\;.
\end{equation*} 
\end{proposition}
\begin{proof}
Observe that 
\begin{equation*}
 \W(X^n)
 = \sum_{k=0}^n \binom{n}{k} \mu_\sX^{-1}(X^k) X^{n-k} 
 = (\mu_\sX^{-1} \ast \id)(X^n)\;.
\end{equation*} 
Therefore, \Cref{thm:conv} implies 
\begin{equation*}
 \Lambda(\W)(t) 
 = \Lambda(\mu_\sX^{-1} \ast \id)(t)
 = \Lambda(\mu_\sX^{-1})(t)\Lambda(\id)(t)
 = \frac{\Lambda(\id)(t)}{\Lambda(\mu_\sX)(t)}
 = \frac{\e^{tX}}{\expec{\e^{t\sX}}}\;,
\end{equation*} 
where we have used \Cref{cor:Lambda} and~\eqref{eq:moment_generating}. 
\end{proof}

Note in particular that $\expec{\W(t,\sX)} = 1$. This is a form of orthogonality relation, 
as it shows that $\expec{\W(\sX^n)} = 0$ for all $n\geqs1$. 

\begin{example}
If $\sX$ is centered Gaussian with variance $\sigma^2$, than $\kappa_\sX(X^2) = \sigma^2$, 
while $\kappa_\sX$ vanishes everywhere else. Therefore, $\W(t,X) = \e^{tX-\sigma^2t^2/2}$ 
is the generating function of Hermite polynomials, and~\eqref{eq:W(Xn)} yields
\begin{equation*}
 \W(X^n) = n! \sum_{k=0}^{\lfloor n/2\rfloor} \frac{(-1)^k}{k!(n-2k)!2^k} \sigma^{2k} X^{n-2k}\;,
\end{equation*} 
which is the $n$th Hermite polynomial with variance $\sigma^2$. 
\end{example}

\subsection{Bell polynomials}
\label{ssec:Bell} 

We consider now the case where the cumulants are given by 
\begin{equation*}
 \kappa_\sX(X^n) = 
 \begin{cases}
  0 & \text{if $n = 1$\;,}\\
  Y_n & \text{if $n \geqs 2$\;,}
 \end{cases}
\end{equation*} 
where the $Y_n$ are for now considered as real parameters. If we assume that only finitely 
many $Y_n$ are different from zero, all sums will be well-defined.

\begin{remark}
Like in~\cite{EFPTZ18}, we work with \emph{formal} power series and do not deal with issues 
of convergence. In particular, the above definition of~$\kappa_\sX(X^n)$ is not in contradiction 
to~Marcinkiewicz' theorem about the cumulants of real-valued random variables.
Besides, we will only be interested in the case when the~$Y_n$'s are elements in an algebra, 
see \Cref{ssec:algebra_valued} below.
\end{remark}

Then we have 
\begin{equation*}
 \W(t,X) = \e^{tX - K_\sX(t)}
 = \exp\biggset{tX - \sum_{n\geqs2} Y_n\frac{t^n}{n!}}\;,
\end{equation*} 
which is the generating function of the exponential Bell polynomials. Namely,
\begin{equation*}
 \W(t,X) = \sum_{n=0}^\infty B_n(X,-Y_2,\dots,-Y_n) \frac{t^n}{n!}\;,
\end{equation*} 
where $B_n$ is by definition the $n$th complete exponential Bell polynomial. 

\begin{lemma}
\label{lem:Bell} 
The $n$th complete exponential Bell polynomial can be written 
\begin{equation}
\label{eq:Bell2} 
 B_n(X,-Y_2,\dots,-Y_n)
= n!
\sum_{\substack{j_1,\dots,j_n\geqs0\\j_1+2j_2+\dots +nj_n = n}}
\frac{1}{j_1!} \biggpar{\frac{X}{1!}}^{j_1} 
\prod_{p=2}^n \frac{1}{j_p!} \biggpar{\frac{-Y_p}{p!}}^{j_p}\;.
\end{equation} 
\end{lemma}
\begin{proof}
By~\eqref{eq:W(Xn)}, and since $\kappa_\sX(X) = 0$, we have 
\begin{equation}
B_n(X,-Y_2,\dots,-Y_n)
= \W(X^n) 
= n!\sum_{k=0}^n \sum_{j=1}^k \frac{(-1)^j}{j!}
 \sum_{\substack{k_1,\dots,k_j\geqs2\\k_1+\dots+k_j=k}}
 \frac{Y_{k_1}}{k_1!} \dots  \frac{Y_{k_j}}{k_j!} \frac{X^{n-k}}{(n-k)!}\;.
\label{eq:Bell1} 
\end{equation}
We may rewrite this expression in a different way, by ordering terms according to the 
value of the $k_i$. For $p\geqs2$, let $j_p$ denote the number of indices $k_i$ equal to $p$
(this number may be $0$). For a given value of $k$ and $j$, the sum of all $j_p$ has to be equal 
to $j$, while the condition $k_1+\dots+k_j=k$ translates into $2j_2 + 3j_3 + \dots + kj_k = k$. 
It is also important to note that in~\eqref{eq:Bell1}, permutations of the $k_i$ are allowed, 
and count as different terms. As a result, we have 
\begin{align*}
    B_n(X,-Y_2,\dots,-Y_n)
    = & n!\sum_{k=0}^n \sum_{j=1}^k \frac{(-1)^j}{j!} \times \\
    & \quad \times \sum_{\substack{j_2,\dots,j_k\geqs0\\j_2+j_3+\dots +j_k = j\\2j_2+3j_3+\dots +kj_k = k}}
    \binom{j}{j_2\dots j_k} \biggpar{\frac{Y_2}{2!}}^{j_2} \dots \biggpar{\frac{Y_k}{k!}}^{j_k}
    \frac{X^{n-k}}{(n-k)!}\;,
\end{align*}
where the multinomial coefficient accounts for the permutations of the $k_i$. Setting 
$j_1 = n-k$ and distributing the factor $(-1)^j$ over the $Y_p$, the sum over $j$ and 
the condition on the sum of the $j_p$ can be dropped, 
leading to~\eqref{eq:Bell2}. 
\end{proof}

The complete Bell polynomial can be decomposed as a sum of incomplete Bell polynomials
according to the sum of the $j_p$. Namely, one has 
\begin{equation*}
 B_n(X,-Y_2,\dots,-Y_n) 
 = \sum_{k=1}^n B_{n,k}(X,-Y_2,\dots,-Y_{n-k+1})\;, 
\end{equation*} 
where 
\begin{equation*}
 B_{n,k}(X,-Y_2,\dots,-Y_{n-k+1})
= n!
\sum_{\substack{j_1,\dots,j_{n-k+1}\geqs0\\
j_1+\dots+j_{n-k+1}=k\\j_1+2j_2+\dots +(n-k+1)j_{n-k+1} = n}}
\frac{1}{j_1!} \biggpar{\frac{X}{1!}}^{j_1} 
\prod_{p=2}^{n-k+1} \frac{1}{j_p!} \biggpar{\frac{-Y_p}{p!}}^{j_p}\;.
\end{equation*} 
The Bell polynomials have a simple combinatorial interpretation. The coefficients of 
$B_{n,k}$ count the number of partitions of a set of cardinality $n$ into $k$ subsets, 
where the sizes of the subsets is encoded into the monomial. For instance, 
\begin{equation*}
 B_{5,3}(X,Y_2,Y_3) = 15 X Y_2^2 + 10 X^2 Y_3
\end{equation*} 
means that there are $15$ ways of partitioning a set of $5$ elements into $3$ subsets 
of sizes $1$, $2$ and $2$, and $10$ ways of partitioning it into $3$ subsets of sizes $1$, $1$ 
and $3$. Another interpretation is in terms of substitutions $X^2\mapsto Y_2$, $X^3\mapsto Y_3$, 
\dots, $X^n\mapsto Y_n$: then, $B_n(X,Y_2,\dots,Y_n)$ is obtained by applying these substitutions in all 
possible ways to the monomial $X^n$. 

\subsection{Free algebra}
\label{ssec:free_algebra} 

One drawback of the above definition of the Wick map $\W$ is that $\mu_\sX$ is not a character 
(in gene\-ral, $\mu_\sX(X^n X^m) \neq \mu_\sX(X^n) \mu_\sX(X^m)$), and therefore $\mu_\sX^{-1}$ 
cannot be expressed in terms of the antipode of the polynomial Hopf algebra $\R[X]$. 

This problem can be fixed by lifting $\mu_\sX$ to the symmetric Hopf algebra 
$H = S(\R[X])$, which is the free commutative algebra over $\R[X]$. We will use 
the symbol $\odot$ to denote the product in $H$. The algebra $H$ being free means 
that $X^n\odot X^m$ is an element of $H$ different from $X^{n+m}$. The map $\mu_\sX$ 
lifts in a unique multiplicative way to a map $\hat\mu: H\to\R$. 

We denote the antipode on $H$ by $S_H:H\to H$. \emph{Takeuchi's formula} \cite{Tak73} states that 
\begin{equation*}
 S_H(X^n) = \sum_{k=1}^n (-1)^k \sum_{\substack{n_1,\dots,n_k\geqs1\\n_1+\dots+n_k=n}}
 \frac{n!}{n_1!\dots n_k!} X^{n_1} \odot \dots \odot X^{n_k}\;.
\end{equation*} 
The inverse of $\hat\mu_\sX$ is then given by $\hat\mu_\sX^{-1} = \hat\mu_\sX \circ S_H$, 
which is compatible with~\eqref{eq:phi_inverse}. As a consequence, the following 
diagram commutes, where $\iota$ denotes the canonical injection and $\Delta_H$ 
denotes the coproduct on $H$:

\begin{equation}  
\label{comdiag:free}
\begin{tikzcd}[column sep=large, row sep=large]
\R[X]
\arrow[r, "\iota", hook] 
\arrow[d, "\Delta"']
\arrow[ddd, bend right=60, "\W"']
& H
\arrow[d, "\Delta_H"] 
\\
\displaystyle
\R[X] \otimes \R[X]
\arrow[r, "\iota\otimes\iota", hook] 
\arrow[dd, "\mu_\sX^{-1}\otimes\id"']
& H\otimes H
\arrow[d, "S_H\otimes\id"] 
\\
& H\otimes H 
\arrow[d, "\hat\mu_\sX\otimes\id"] \\
\R[X] 
\arrow[r, "\iota", hook] 
& H
\end{tikzcd}
\end{equation}

We have explicitly 
\begin{equation*}
 (S_H\otimes\id)\Delta_H(X^n) 
 = n!\sum_{k=0}^n \sum_{j=1}^k (-1)^j 
 \sum_{\substack{k_1,\dots,k_j\geqs1\\k_1+\dots+k_j=k}}
 \frac{X^{k_1}}{k_1!} \odot \dots \odot \frac{X^{k_j}}{k_j!} \otimes \frac{X^{n-k}}{(n-k)!}\;.
\end{equation*} 
As in the proof of \Cref{lem:Bell}, we can rewrite this expression in terms of 
the number $j_p$ of indices $k_i$ equal to $p$, and distributing the factor $(-1)^j$ 
over the powers of $X$. The result is 
\begin{equation*}
 (S_H\otimes\id)\Delta_H(X^n) 
 = n!\sum_{k=0}^n \sum_{j=1}^k j!
 \sum_{\substack{j_1,\dots,j_k\geqs0\\j_1+j_2+\dots+j_k=j\\j_1+2j_2+\dots+kj_k=k}}
 \bigodot_{p=1}^j \frac{1}{j_p!} \biggpar{\frac{-X^p}{p!}}^{\odot j_p}
 \otimes \frac{X^{n-k}}{(n-k)!}\;.
\end{equation*} 
Note the factor $j!$, which will disappear when applying $\hat\mu_\sX\otimes\id 
= \exp_\ast(\hat\kappa_\sX)\otimes\id$ to this expression to recover~\eqref{eq:Bell2}. 

\subsection{Algebra-valued moments}
\label{ssec:algebra_valued} 

The above setting can be extended to the case where the map $\mu_\sX$ takes values 
in an arbitrary commutative algebra $A$. This can be done by considering an 
\emph{extension of scalars}: consider the space $A[X] = A\otimes\R[X]$, endowed 
with a left $A$-module structure given by $a\cdot(b\otimes P) = ab\otimes P$. 
It can be thought of as consisting of polynomials with coefficients in $A$. 

For any algebra $A$, the space of linear maps $\sL(\R[X],A)$ is an algebra for 
the convolution product
\begin{equation*}
 \ph \ast \psi = m_A (\ph\otimes\psi) \Delta\;,
\end{equation*} 
where $m_A$ is the multiplication map, defined by $m_A(a\otimes b) = ab$. 
Its unit element is $u_A\circ\eps: \R[X]\to A$, where $u_A$ denotes the projection 
on the unit of $A$. 

All the properties of the previous sections hold in this setting as well. In particular, 
the Wick exponential~\eqref{eq:Wick_expo} is now a map $\W:\R[X]\to A[X]$. It satisfies 
\begin{equation*}
 m_A \circ \mu^A_\sX \circ \W = u_A\circ \eps\;,
\end{equation*} 
where $\mu^A_\sX = \id\otimes\mu_\sX: A[X]\to A$ is the extension of $\mu_\sX$. 

We will work in the particular setting where $A = \R[Y]$ is the polynomial algebra in a
single variable $Y$. Then one can identify $A[X] = \R[Y]\otimes\R[X]$ with $\R[X,Y]$, 
via the identification of $Y^m\otimes X^n$ with $X^nY^m$. Note that we have 
\begin{equation*}
 \mu^A_\sX(Y^m\otimes X^n) = Y^m \mu_\sX(X^n)\;.
\end{equation*}

%% file: multiindices.tex
\section{Multi-indices}
\label{sec:multiindices} 

The combinatorics of Feynman diagrams becomes quite involved as their size increases. 
It can be simplified, however, by working with multi-indices instead of graphs. 
Multi-indices, introduced in~\cite{LOT_2023}, are monomials encoding information on 
the arity of a graph. A single 
multi-index corresponds in general to a linear combination of several different graphs, 
thereby reducing the complexity of the combinatorics, while keeping information that 
is essential for BPHZ renormalisation. We introduce these objects by 
summarising material from~\cite{bruned2025renormalisingfeynmandiagramsmultiindices}. 
The key result in this section is \Cref{prop:coproduct_multiindex}, which 
provides an explicit expression for the coproduct of the multi-indices corresponding 
to monomials $X^n$. 


\subsection{Definition of multi-indices}
\label{ssec:multiind} 

We introduce here some definitions and results 
from~\cite{bruned2025renormalisingfeynmandiagramsmultiindices}.

We fix a set of abstract variables $(z_k)_{k\in\N}$. Given a map $\beta:\N\to\N$ 
with finite support, mea\-ning that the number of non-zero values of $\beta$ is finite, 
the associated \emph{multi-index} is the monomial 
\begin{equation*}
 z^\beta = \prod_{k\in\N} z_k^{\beta(k)}\;.
\end{equation*} 
Let $\Gamma = (\sV,\sE) \in\F$ be a connected Feynman diagram. We associate to it the 
multi-index 
\begin{equation}
\label{eq:def_counting_map} 
 \Phi(\Gamma) = \prod_{v\in\cV} z_{k(v)}\;,
\end{equation} 
where $k(v)$ denotes the \emph{arity} of the vertex $v$, that is, the number of 
edges adjacent to $v$. The map $\Phi$ is called \emph{counting map}. For instance, 
we have 
\begin{equation*}
 \Phi\bigpar{\FGIII} = z_3^2\;, \qquad 
 \Phi\Bigpar{\FGIIIplus} = z_2 z_4^2\;.
\end{equation*} 
Note that in our situation, $\beta(k)$ can only differ from $0$ for $k\in\set{2,3,4}$. 
The length and the degree of a multi-index are defined, respectively, by 
\begin{equation*}
 \bigabs{z^\beta} = \sum_{k\in\N} \beta(k)\;, \qquad
 \deg(z^\beta) = d \bigpar{\bigabs{z^\beta}-1}
 -\frac{d-2}{2} \sum_{k\in\N} k\beta(k)\;.
\end{equation*} 
Multi-indices with non-positive degree are called \emph{divergent}.
The degree is preserved by the counting map, that is, 
\begin{equation*}
 \deg\bigpar{\Phi(\Gamma)} = \deg(\Gamma)
 \qquad \forall\Gamma\in\F\;
\end{equation*} 
where the degree for a Feynman diagram~$\Gamma$ has been introduced in~\eqref{e:deg_graphs} above.
In addition, a multi-index has a symmetry factor, defined by 
\begin{equation*}
 S_\tM(z^\beta) = \prod_{k\in\N} \beta(k)!(k!)^{\beta(k)}\;.
\end{equation*} 
A Feynman diagram $\Gamma$ also has a symmetry factor, defined as 
\begin{equation*}
 S_\tF(\Gamma) = \bigabs{\Aut(\Gamma)}\;,
\end{equation*} 
where $\Aut(\Gamma)$ denotes the \emph{automorphism group} of $\Gamma$, that is, the 
permutations of vertices and edges of $\Gamma$ that leave the diagram invariant 
(see~\cite[Definition~4.1]{bruned2025renormalisingfeynmandiagramsmultiindices} for details). 
This allows to define the inverse map of~\eqref{eq:def_counting_map}, given by 
\begin{equation}
\label{eq:PM} 
 \sP_\tM(z^\beta) = \sum_{\Gamma\colon\Phi(\Gamma)=z^\beta}
 \frac{S_\tM(z^\beta)}{S_\tF(\Gamma)} \Gamma\;, 
\end{equation} 
see~\cite[Definition~3.2 and Proposition~3.4]{bruned2025renormalisingfeynmandiagramsmultiindices}.

In an analogous way to what we did for Feynman diagrams, we introduce the following spaces:
\begin{itemize}
    \item $\M$ is the set of non-empty multi-indices (meaning that the $\beta(k)$ cannot all be zero), which also belong to the image $\Phi(\sF)$,
    and $\spanM$ is its linear span; 
    \item $\Mminus \subset \M$ denotes the divergent multi-indices in~$\M$, and $\spanMminus$ its linear span;  
    \item $\sM$ is the algebra generated by~$\M$ w.r.t. the forest product~$\odot$, and~$\spansM$ is its linear span;
    \item $\sMminus \subset \sM$ is the subalgebra of~$\sM$ generated by~$\Mminus$;
    we also let~$\spansMminus$ denote the span of~$\sMminus$. 
\end{itemize}
We will denote the neutral element w.r.t.~$\odot$ by $\unit_\tM$,
and by $\Pi_\tM$ a valuation map on multi-indices. The objects defined 
above can be extended to forests in $\sM$. For instance, 
\begin{align*}
 \Phi(\Gamma_1\cdot \ldots \cdot \Gamma_n) 
 &= \Phi(\Gamma_1) \odot \ldots \odot \Phi(\Gamma_n)\;, \\
 \bigabs{z^{\beta_1} \odot \ldots \odot z^{\beta_n}}
 &= \bigabs{\beta_1} + \dots + \bigabs{\beta_n}\;, \\
 \deg\bigpar{z^{\beta_1}\odot\ldots\odot z^{\beta_n}}
 &= \deg\bigpar{\beta_1}  + \dots + \deg\bigpar{\beta_n}\;,
\end{align*}
while the symmetry factor of a forest of multi-indices is defined as 
\begin{equation*}
 S_\tM\bigpar{(z^{\beta_1})^{\odot r_1}\odot\ldots\odot (z^{\beta_n})^{\odot r_n}}
 = \prod_{i=1}^n 
 r_i! \bigpar{S_\tM(z^{\beta_i})}^{r_i}\;,
\end{equation*} 
where the $\beta_i$ are all assumed to be different from each other. 

In~\cite{bruned2025renormalisingfeynmandiagramsmultiindices}, the authors also construct 
a coproduct $\Delta_\tM$ and a twisted antipode $\tilde\sA_\tM$ such that the following diagram 
commutes (see~\cite[Theorem~4.9]{bruned2025renormalisingfeynmandiagramsmultiindices}):
\begin{equation}  
\label{comdiag:M-G}
\begin{tikzcd}[column sep=large, row sep=large]
\spansM
\arrow[r, "\sP_\tM"] 
\arrow[d, "\Delta_\tM"']
& \spansF
\arrow[d, "\DeltaCK"] 
\\
\displaystyle
\spansMminus \otimes \spansM
\arrow[r, "\sP_\tM\otimes\sP_\tM"] 
\arrow[d, "\Pi_\tM\tilde{\sA}_\tM\otimes\id"']
& \spansFminus\otimes \spansF
\arrow[d, "\Pi_N\tilde{\sA}\otimes\id"] 
\\
\spansM
\arrow[r, "\sP_\tM"] 
& \spansF
\end{tikzcd}
\end{equation}
The structure of the renormalisation map on Feynman diagrams is thus exactly 
mirrored on the space of multi-indices. Note that this diagram remains 
valid if all instances of $\spansM$ are replaced by $\spanM$, and all 
instances of $\spansF$ are replaced by $\spanF$.


\subsection{Subdivergences}
\label{ssec:subdiv}

Within the setting of multi-indices, it is easy to identify the structure of divergent 
subdiagrams (or \emph{subdivergences}), as we do in the following lemma.
Recall that~$\Fminus$ denotes the set of divergent \emph{connected} Feynman diagrams 
which, as is our standing assumption, have no free legs (called vacuum diagrams).

\begin{lemma}
\label{lem:subdivergences} 
Let~$\Gamma \in \Fminus$. Then we necessarily have
\begin{equation*}
\Phi(\Gamma) \in \left\{z_2z_4^{p-1}\right\}_{p \geqs3}\bigcup \left\{z_3^2 z_4^{p-2}\right\}_{p \geqs2} \,.
\end{equation*}
In other words: The only potentially divergent subdiagrams are those having multi-index 
$z_2z_4^{p-1}$ with $p\geqs3$ or $z_3^2 z_4^{p-2}$ with $p\geqs2$. 
\end{lemma}
\begin{proof}
Let $\beta$ be such that $\beta(k) = 0$ for $k\not\in\set{2,3,4}$. Then 
\begin{align*}
\deg(z^\beta) 
&= -\frac{d-2}{2} \sum_{k=2}^4 k\beta(k) + d \biggpar{\sum_{k=2}^4 \beta(k) - 1} \\
&= 2\beta(2) + \biggpar{3-\frac{d}{2}}\beta(3) + (4-d)\beta(4) - d\;.
\end{align*}
This is a decreasing function of $d$. In particular, in the limiting case $d=4$, we 
get 
\begin{equation*}
\deg(z^\beta) = 2\beta(2)+\beta(3) - 4\;,
\end{equation*} 
showing that $z^\beta$ can only be divergent if $2\beta(2)+\beta(3) < 4$. 
In addition, since we are considering vacuum diagrams, the number of half-edges of the associated diagram should be even, 
which imposes $\beta(3)\in\set{0,2}$. The only options are $\beta(3) = 0$ and 
$\beta(2) = 1$, leading to $z_2z_4^{p-1}$, and $\beta(3) = 2$ and $\beta(2) = 0$, 
leading to $z_3^2 z_4^{p-2}$. 
\end{proof}

\Cref{tab:subdivergence} lists the first subdivergences appearing as the dimension 
$d$ increases. 

\begin{table}[h]
  \centering
\begin{tabular}{|c|c|c|c|c|}
\hline
Graph & Multi-index & Degree & Critical $d$ & Minimal $n$ \\
\hline
\vrule height 16pt depth 8pt width 0pt
$\FGIII$ & $z_3^2$ & $6 - 2d$ & $3 = d^*_{\text{m}}(2)$ & $4$ \\
\hline
\vrule height 20pt depth 8pt width 0pt
$\FGTwoTwoOne$ & $z_3^2z_4$ & $10 - 3d$ & $\frac{10}3 = d^*_{\text{m}}(3)$ & $5$ \\ 
\vrule height 16pt depth 8pt width 0pt
$\FGIIIplus$ & $z_2z_4^2$ &  &  &  \\ 
\hline
\vrule height 20pt depth 8pt width 0pt
$\FGQuadA$ $\FGQuadB$ $\FGQuadD$ & $z_3^2z_4^2$ & $14 - 4d$ & $\frac72 = d^*_{\text{m}}(4)$ & $6$ \\ 
\vrule height 16pt depth 12pt width 0pt
$\FGQuadC$ & $z_2z_4^3$ &  &  &  \\ 
\hline
\end{tabular}
\caption{List of the first divergent subdiagrams, with their multi-index, degree, 
value of $d$ for which they become divergent, and minimal value of $n$ such that they 
occur in $\sP(X^n)$.}
\label{tab:subdivergence} 
\end{table}


\subsection[Computation of $\Delta_\tM(z_4^n)$]{Computation of $\boldsymbol{\Delta_\tM(z_4^n)}$}
\label{ssec:Delta_M}

We will write $\mathring\Delta_\tM$ for the \emph{reduced coproduct}, which is such that 
\begin{equation*}
 \Delta_\tM(z^\beta) = \unit_\tM \otimes z^\beta 
 + z^\beta \otimes \unit_\tM + \mathring\Delta_\tM(z^\beta) \,.
\end{equation*} 
(Note that in~\cite{bruned2025renormalisingfeynmandiagramsmultiindices}, the authors 
write $\Delta_\tM$ for the reduced coproduct, and $\Delta_\tM^-$ for the full coproduct). 
A simplification arises from the fact that in BPHZ renormalisation, only terms with 
non-positive degree on the left of the tensor product play a role. We can therefore 
restrict the coproduct to these terms, and we will use the same notation for that 
restricted form. 

The following result establishes an explicit expression for the reduced coproduct 
of monomials $z_4^n$. 

\begin{proposition}
\label{prop:coproduct_multiindex} 
For any $p\geqs 2$, define 
\begin{equation}
\label{eq:def_Yp} 
 \sY_p = 
 \begin{cases}
  16 z_3^2 & \text{if $p = 2$\;,} \\
  6pz_2z_4^{p-1} + 8p(p-1)z_3^2z_4^{p-2} & \text{if $p \geqs 3$\;.}
 \end{cases}
\end{equation} 
Then for any $n\geqs4$, one has 
\begin{equation}
\label{eq:DeltaM_z4} 
 \mathring\Delta_\tM(z_4^n)
 = n! \sum_{k=2}^{n-1} \sum_{\substack{j_1,\dots,j_{n-2}\geqs0\\j_1 + j_2 + \dots + j_{n-2} = k\\ 
 j_1 + 2j_2 + \dots + (n-2)j_{n-2} = n}} 
 \bigodot_{p=2}^{n-2} \frac{1}{j_p!} \biggpar{\frac{\sY_p}{p!}}^{\odot j_p} 
 \otimes \frac{1}{j_1!} z_2^{k-j_1} z_4^{j_1}\;,
\end{equation} 
where the sum ranges over non-negative integers $j_p$. 
\end{proposition}

We will apply~\cite[Proposition~3.7]{bruned2025renormalisingfeynmandiagramsmultiindices}
to the particular case $z^\beta = z_4^n$. The general expression for the reduced 
coproduct is
\begin{equation}
\label{eq:coprod_z4n} 
 \mathring\Delta_\tM(z^\beta) = 
 \sum_{z^{\beta_1} \odot \ldots \odot z^{\beta_m} \in \cM}
 \sum_{z^\alpha\in\mathbf{M}}
 E(z^{\beta_1} \odot \ldots \odot z^{\beta_m}, z^\alpha, z^\beta) \,
 \bigpar{z^{\beta_1} \odot \ldots \odot z^{\beta_m}}\otimes z^\alpha\;, 
\end{equation} 
where the coefficients $E(\cdot, \cdot, \cdot)$ are given by 
\begin{align}
\thinspace & E(z^{\beta_1} \odot \ldots \odot z^{\beta_m}, z^\alpha, z^\beta) \label{eq:defE}\\
 & = \sum_{k_1,\dots,k_m\in\N}
 \sum_{\hat\beta_1 + \dots + \hat\beta_m + \hat\alpha = \beta}
 \frac{S_\tM(z^\beta)}{S_\tM(z^{\beta_1} \odot \ldots \odot z^{\beta_m})S_\tM(z^\alpha)} \ \times \\[0.5em]
 & \qquad \qquad \times
 \frac{\pscal{\prod_{i=1}^m\partial_{z_{k_i}}z^{\alpha}}
 {z^{\hat\alpha}}}{S_\tM(z^{\hat\alpha})}
 \prod_{i=1}^m \frac{\pscal{D^{k_i}z^{\beta_i}}{z^{\hat\beta_i}}}{S_\tM(z^{\hat\beta_i})}\;. \notag
\end{align} 
Here the inner product for multi-indices is defined by 
\begin{equation}
\label{eq:pscal_multiindex} 
 \pscal{z^\alpha}{z^\beta} = S_\tM(z^\alpha) \delta^\alpha_\beta\;, 
\end{equation} 
while $D$ is the linear map 
\begin{equation*}
 D = \sum_{k\in\N} z_{k+1} \partial_{z_k}\;.
\end{equation*} 
In what follows, we write exponents of multi-indices in the form 
$\beta = [\beta(2),\beta(3),\beta(4)]$, since all other $\beta(k)$ vanish in our 
situation. We need to apply~\eqref{eq:coprod_z4n} to $\beta = [0,0,n]$. 
Since we restrict the coproduct to terms that are divergent on the left side 
of the tensor product, \Cref{lem:subdivergences} implies that the 
$\beta_i$ are of the form $[1,0,p-1]$ with $p\geqs3$ or $[0,2,p-2]$ with $p\geqs2$.

\begin{lemma}
\label{lem:E} 
Denote by $a_p$ the number of $z^{\beta_i}$ equal to $z_2z_4^{p-1}$, and by $b_p$ 
the number of $z^{\beta_i}$ equal to $z_3^2 z_4^{p-2}$, with the convention $a_2 = 0$
to avoid case distinctions between the $a_p$ and $b_p$. Then the only non-vanishing 
coefficient in~\eqref{eq:coprod_z4n} for these $\beta_i$ is 
\begin{equation*}
 E(z^{\beta_1}\odot\ldots\odot z^{\beta_m}, z_2^m z_4^{n-s}, z_4^n)
 = \frac{n!}{(n-s)!}
 \prod_{p=2}^{n-2} \frac{1}{a_p!} \biggpar{\frac{4!}{4(p-1)!}}^{a_p} 
 \frac{1}{b_p!} \biggpar{\frac{4!}{3(p-2)!}}^{b_p}\;,
\end{equation*} 
where 
\begin{equation*}
 s = \sum_{p=2}^{n-2} pj_p\;, \qquad 
 j_p = a_p + b_p\;.
\end{equation*} 
\end{lemma}
\begin{proof}
A direct computation shows that 
\begin{align}
D(z_2 z_4^{p-1}) &= z_3z_4^{p-1} + R_1\;, 
&
D(z_3^2 z_4^{p-2}) &= 2z_3z_4^{p-1} + R_3\;, \\
D^2(z_2 z_4^{p-1}) &= z_4^p + R_2\;, 
&
D^2(z_3^2 z_4^{p-2}) &= 2z_4^p + R_4\;, 
\label{eq:Dzz} 
\end{align}
where the $R_i$ are residual terms that vanish when $z_k = 0$ for all $k\geqs5$.
All higher derivatives also vanish when $z_k = 0$ for all $k\geqs5$.

We now observe that the condition $\hat\beta_1 + \dots + \hat\beta_m + \hat\alpha
= \beta = [0,0,n]$ implies that all $\hat\beta_i$, as well as $\hat\alpha$, are 
of the form $[0,0,\star]$. It follows from~\eqref{eq:pscal_multiindex} and~\eqref{eq:Dzz} 
that $\pscal{D^{k_i}z^{\beta_i}}{z^{\hat\beta_i}} = 0$ unless $k_i = 2$ and 
$\hat\beta_i = [0,0,p]$, in which case 
\begin{equation*}
 \frac{\pscal{D^2z^{\beta_i}}{z^{\hat\beta_i}}}{S_\tM(z^{\hat\beta_i})}
 = 
 \begin{cases}
  1 & \text{if $\beta_i = [1,0,p-1]$\;,} \\
  2 & \text{if $\beta_i = [0,2,p-2]$\;.}
 \end{cases}
\end{equation*} 
Writing 
\begin{equation*}
 r = \sum_{p=2}^{n-2} b_p\;, \qquad 
 m-r = \sum_{p=2}^{n-2} a_p\;,
\end{equation*} 
we obtain 
\begin{equation}
\label{eq:prod1} 
 \prod_{i=1}^m \frac{\pscal{D^2z^{\beta_i}}{z^{\hat\beta_i}}}{S_\tM(z^{\hat\beta_i})}
 = 2^r\;.
\end{equation} 
Furthermore, we have
\begin{equation*}
 \hat\beta_1 + \dots + \hat\beta_m = 
 [0,0,\sum_{p=2}^{n-2}p j_p] = 
 [0,0,s]\;,
\end{equation*} 
and therefore the condition $\hat\beta_1 + \dots + \hat\beta_m + \hat\alpha
= \beta = [0,0,n]$ imposes 
\begin{equation*}
 \hat\alpha = [0,0,n-s]\;,
\end{equation*} 
or, equivalently, $z^{\hat\alpha} = z_4^{n-s}$. Therefore, since all $k_i$ 
are equal to $2$, we obtain 
\begin{equation}
\label{eq:prod2} 
 \frac{\pscal{\prod_{i=1}^m\partial_{z_{k_i}}z^{\alpha}}
 {z^{\hat\alpha}}}{S_\tM(z^{\hat\alpha})}
 = m! \mathbf{1}_{\alpha = [m,0,n-s]}
\end{equation} 
showing that we necessarily have $z^\alpha = z_2^m z_4^{n-s}$. 
It remains to compute some symmetry factors. First, we compute 
\begin{equation}
\label{eq:prod3} 
 \frac{S_\tM(z_4^n)}{S_\tM(z^\alpha)}
 = \frac{n!(4!)^n}{m!(2!)^m(n-s)!(4!)^{n-s}}
 = \frac{n!(4!)^s}{m!(n-s)!(2!)^m}\;.
\end{equation} 
Next, since 
\begin{align*}
 S_\tM(z_2z_4^{p-1}) &= 2!(p-1)!(4!)^{p-1}\;, \\
 S_\tM(z_3^2z_4^{p-2}) &= 2!(3!)^2(p-2)!(4!)^{p-2} 
 = 3(p-2)!(4!)^{p-1}\;,
\end{align*}
we get 
\begin{equation}
\label{eq:prod4} 
 S_\tM\bigpar{z^{\beta_1}\odot\ldots\odot z^{\beta_m}}
 = \prod_{p=2}^{n-2} a_p! \Bigbrak{2!(p-1)!(4!)^{p-1}}^{a_p} 
 b_p! \Bigbrak{3(p-2)!(4!)^{p-1}}^{b_p}\;.
\end{equation} 
Plugging~\eqref{eq:prod1}, \eqref{eq:prod2}, \eqref{eq:prod3} and~\eqref{eq:prod4} 
into~\eqref{eq:defE}, we finally obtain 
\begin{align*}
  E(z^{\beta_1}\odot\ldots\odot & z^{\beta_m}, z_2^m z_4^{n-s}, z_4^n)
 = \frac{n!}{(n-s)!} \frac{(4!)^s2^r}{(2!)^m} \times \\[0.5em]
 & \qquad \times \prod_{p=2}^{n-2} \frac{1}{a_p! \bigbrak{2!(p-1)!(4!)^{p-1}}^{a_p} 
 b_p! \bigbrak{3(p-2)!(4!)^{p-1}}^{b_p}}\;,
\end{align*} 
from which the result follows, upon expanding $(4!)^s = \prod_p 4!^{p(a_p+b_p)}$,
$2^r = \prod_p 2^{b_p}$ and finally $(2!)^m = \prod_p (2!)^{a_p + b_p}$. 
\end{proof}

\begin{proof}[\textsc{Proof of \Cref{prop:coproduct_multiindex}}]
\Cref{lem:E} and~\eqref{eq:coprod_z4n} imply 
\begin{equation*}
 \mathring\Delta_\tM(z^\beta) = 
 \sum_{s,m} F^{(n)}_{s,m} \otimes z_2^m z_4^{n-s}\;,
\end{equation*} 
where 
\begin{equation*}
 F^{(n)}_{s,m} = 
 \frac{n!}{(n-s)!} \sum_{(a_p,b_p)}
 \bigodot_{p=2}^{n-2} \frac{1}{a_p!} \biggpar{\frac{4!z_2z_4^{p-1}}{4(p-1)!}}^{\odot a_p} 
 \frac{1}{b_p!} \biggpar{\frac{4!z_3^2z_4^{p-2}}{3(p-2)!}}^{\odot b_p}\;,
\end{equation*} 
with the sum running over all $(a_p,b_p)$ satisfying 
\begin{equation*}
 \sum_{p=2}^{n-2} j_p = m\;, \qquad 
 \sum_{p=2}^{n-2} pj_p = s\;, \qquad 
 j_p = a_p + b_p\;.
\end{equation*} 
We first perform the sum over all $(a_p,b_p)$ summing to $j_p$. The binomial 
formula yields 
\begin{equation*}
 \sum_{a_p+b_p = j_p} 
 \frac{1}{a_p!} \biggpar{\frac{4!z_2z_4^{p-1}}{4(p-1)!}}^{\odot a_p} 
 \frac{1}{b_p!} \biggpar{\frac{4!z_3^2z_4^{p-2}}{3(p-2)!}}^{\odot b_p}
 = \frac{1}{j_p!} \biggpar{\frac{\sY_p}{p!}}^{\odot j_p}\;,
\end{equation*} 
where the $\sY_p$ are given in~\eqref{eq:def_Yp}. This implies 
\begin{equation*}
 F^{(n)}_{s,m} = 
 \frac{n!}{(n-s)!} \sum_{\substack{j_2+\dots+j_{n-2}=m\\2j_2+\dots+(n-2)j_{n-2}=s}} 
 \bigodot_{p=2}^{n-2}\frac{1}{j_p!} \biggpar{\frac{\sY_p}{p!}}^{\odot j_p}\;. 
\end{equation*} 
The last step is to change the summation over variables $(s,m)$ to that over 
variables $(k,j_1)$, where $k=n+m-s$ and $j_1 = n-s$. In particular, the sum of 
the $j_p$ starting from $p=1$ is now $k$, while the sum of the $pj_p$ is $n$.
\end{proof}

\begin{remark}
\label{rem:k=1} 
The sum~\eqref{eq:DeltaM_z4} can be extended to all $k$ from $1$ to $n-1$, 
because the conditions on the $j_p$ cannot be both satisfied if $k=1$ since 
$p\leqs n-2$. 
\end{remark}

\begin{remark}
On the right-hand side of~\eqref{eq:DeltaM_z4}, the only $k$-dependence occurs 
in the term $z_2^{k-j_1}$. If the expression is evaluated in $z_2 = 1$, the 
sum over $k$ and the condition $j_1+j_2+\dots+j_{n-2}=k$ can be dropped. 
\end{remark}

\begin{example}
\label{ex:n4} 
Consider the case $n=4$. Then $k$ can take the values $2$ and $3$. For $k=2$, the only 
possible decomposition is $j_1 = 0$, $j_2 = 2$, while for $k=3$, the only option is 
$j_1 = 2$, $j_2 = 1$. Since $\sY_2 = 16z_3^2$, we obtain 
\begin{align*}
 \mathring\Delta_\tM(z_4^4) 
 &= 4! \biggpar{\frac{1}{(2!)^3} \sY_2^{\odot 2} \otimes z_2^2 + \frac{1}{(2!)^2}\sY_2\otimes z_2z_4^2} \\
 &= 4! \bigpar{32 (z_3^2)^{\odot 2} \otimes z_2^2 + 4 z_3^2\otimes z_2z_4^2}\;.
\end{align*} 
We can check commutativity of the upper part of diagram~\eqref{comdiag:M-G}. Applying~\eqref{eq:PM}, 
we find 
\begin{equation*}
 \sP_\tM\bigpar{z_2^2} = 2\, \FGII\;, \qquad
 \sP_\tM\bigpar{z_3^2} = 6\, \FGIII\;, \qquad
 \sP_\tM\bigpar{z_2z_4^2} = 2^6\cdot3\, \FGIIIplus\;,
\end{equation*} 
which implies 
\begin{equation}
\label{eq:PMPMDeltaz44} 
 (\sP_\tM\otimes \sP_\tM)\mathring\Delta_\tM(z_4^4) 
 = 2^{11}\cdot3^3 
 \biggpar{\FGIII^{\cdot2}\otimes\FGII + 2\,\FGIII\otimes\FGIIIplus}\;.
\end{equation} 
On the other hand, one finds
\begin{equation*}
 \sP_\tM(z_4^4) = \sP(X^4) = 2^{11}\cdot3^3\FGXFoursing + \dots\;,
\end{equation*} 
where the dots indicate diagrams without subdivergences. 
Applying $\DeltaCKred$ to this expression indeed yields~\eqref{eq:PMPMDeltaz44}, 
since we can extract one or two \lq\lq sunset diagrams\rq\rq\ $\FGIII$. 

For $n=5$, the result is 
\begin{equation}
\label{eq:ex_n5_Delta} 
 \mathring\Delta_\tM(z_4^5)
 = 5!\biggpar{\frac{1}{2!3!}\sY_2 \odot\sY_3 \otimes z_2^2
 + \frac{1}{2!3!} \sY_3 \otimes z_2z_4^2 
 + \frac{1}{(2!)^3} \sY_2^{\odot 2} \otimes z_2^4z_4 
 + \frac{1}{2!3!} \sY_2 \otimes z_2z_4^3}\;, 
\end{equation} 
where the coefficients $(10,10,15,10)$ are as in the Bell polynomial
\begin{equation}
\label{eq:ex_n5_Bell} 
 B_5(x,y_2,y_3,y_4,y_5) 
 = y_5 + 5y_4x + 10y_2y_3 + 10y_3x^2 + 15y_2^2x + 10y_2x^3 + x^5\;,
\end{equation} 
up to boundary terms. 
\end{example}

\subsection{Extended coproduct}
\label{ssec:extended_coproduct} 

\Cref{ex:n4} above shows some differences between the terms 
appearing in the reduced coproduct (such as~\eqref{eq:ex_n5_Delta}) and 
Bell polynomials (cf.~\eqref{eq:ex_n5_Bell}). These are partly due to the fact 
that the reduced coproduct is used, instead of the full coproduct. However, 
there is also a difference due to the condition $p\leqs n-2$ on the subdivergences 
extracted from $z_4^n$. Therefore, the Bell polynomial~\eqref{eq:ex_n5_Bell} contains a term 
$5y_4x$, while $\mathring\Delta_\tM(z_4^5)$ has no term proportional to 
$\sY_4\otimes z_4$. 

Below, it will be convenient to artificially add this term to the reduced coproduct, 
that is, to consider its extended version 
\begin{equation}
\label{eq:DeltaM_z4_plus} 
 \mathring\Delta_\tM^+(z_4^n)
 = n! \sum_{k=1}^{n-1} \sum_{\substack{j_1,\dots,j_{n-1}\geqs0\\j_1 + j_2 + \dots + j_{n-1} = k\\ 
 j_1 + 2j_2 + \dots + (n-1)j_{n-1} = n}} 
 \bigodot_{p=2}^{n-1} \frac{1}{j_p!} \biggpar{\frac{\sY_p}{p!}}^{\odot j_p} 
 \otimes \frac{1}{j_1!} z_2^{k-j_1} z_4^{j_1}\;,
\end{equation} 
which differs from~\eqref{eq:DeltaM_z4} by the condition on $p$ within the product. 
The following lemma shows that this will not affect the end result; 
see \Cref{rmk:sigma_n} for a reference to the definition of~$\tilde\sA_\tM$.

\begin{lemma}
\label{lem:extended_coproduct}
For all $n\geqs4$, one has 
\begin{equation*}
 \sP_\tM\circ(\Pi_\tM\tilde\sA_\tM \otimes \id) \mathring\Delta_\tM^+(z_4^n) 
 = \sP_\tM\circ(\Pi_\tM\tilde\sA_\tM \otimes \id) \mathring\Delta_\tM(z_4^n)\;.
\end{equation*} 
\end{lemma}
\begin{proof}
The only difference between $\mathring\Delta_\tM^+(z_4^n)$ and $\mathring\Delta_\tM(z_4^n)$ 
is that the former has an additional term, corresponding to $j_{n-1} = 1$. The conditions 
on the $j_p$ impose $j_1 = 1$ while all other $j_p$ vanish, so that $k=2$. One obtains 
\begin{equation*}
 \mathring\Delta_\tM^+(z_4^n) 
 = \mathring\Delta_\tM(z_4^n) + n\sY_{n-1} \otimes z_2z_4\;.
\end{equation*} 
The result follows from the fact that $\sP_\tM(z_2z_4) = 0$.
%
This is because all our considerations are concerned with \emph{vacuum} Feynman diagrams, i.e. those \emph{without} free legs and no such diagram corresponds to~$z_2 z_4$.
\end{proof}

%% file: commutative.tex
\section{Commutative diagram}
\label{sec:comute} 

This section contains the proofs of the main results, \Cref{thm:main} and~\Cref{cor:BPHZ}. 

\subsection{Identification of the counterterms}

We assume now that the algebra-valued map 
$\hat\kappa_\sX: H\to\R[Y]$ is given by 
\begin{equation*}
 \hat\kappa_\sX(X^p) = 
 \begin{cases}
  0 & \text{if $p=1$\;,}\\
  \sigma_p Y & \text{if $p\geqs2$\;,}
 \end{cases}
\end{equation*} 
where the $\sigma_p$ are real numbers. 
The Wick map $\hat\W = (\hat\mu_\sX S_H\otimes\id)\Delta_H: 
H\to \R[Y] \otimes H$ can be written (cf.\ \Cref{lem:Bell}) in the form 
\begin{equation}
\label{eq:identify_Wick} 
 \hat\W(X^n) 
 = n! \sum_{k=0}^n \sum_{\substack{j_1,\dots,j_n\geqs0 \\j_1+\dots+j_n=k \\ j_1+2j_2+\dots+nj_n=n}}
 \prod_{p=2}^n \frac{1}{j_p!} \biggpar{\frac{-\sigma_pY}{p!}}^{j_p} 
 \frac{X^{j_1}}{j_1!} \;.
\end{equation} 
We now define a linear map $\eta:\R[Y] \otimes H \to \spanM$ by 
\begin{equation*}
 \eta\bigpar{Y^m \otimes (X^{n_1}\odot \dots \odot X^{n_k})}
 = z_2^m z_4^{n_1 + \dots + n_k}\;.
\end{equation*} 
Since $Y^{j_2+\dots+j_n} = Y^{k-j_1}$, we have 
\begin{equation}
\label{eq:identify_Wick2} 
 (\eta\circ\hat\W)(X^n) 
 = n! \sum_{k=0}^n \sum_{\substack{j_1,\dots,j_n\geqs0 \\j_1+\dots+j_n=k \\ j_1+2j_2+\dots+nj_n=n}}
 \prod_{p=2}^n \frac{1}{j_p!} \biggpar{\frac{-\sigma_p}{p!}}^{j_p} 
 z_2^{k-j_1} \frac{z_4^{j_1}}{j_1!} \;.
\end{equation} 

\begin{proposition}
Define the $\sigma_p$ by 
\begin{equation}
\label{eq:Yp_PiM} 
 \sigma_p := -\Pi_\tM\tilde{\sA}_\tM(\sY_p)
 = -\Pi_N\tilde{\sA}\sP_\tM(\sY_p) 
\end{equation}
for all $p\geqs2$, where the second equality comes from the commutative diagram~\eqref{comdiag:M-G}. 
Then
\begin{equation}
\label{eq:etaWHXn} 
(\eta\circ\hat\W) (X^n)
 = (\Pi_\tM\tilde{\sA}_\tM \otimes \id) \mathring\Delta_\tM^+(z_4^n) 
 - \sigma_n z_2 + z_4^n
\end{equation} 
holds for all $n\geqs2$. 
\end{proposition}
\begin{proof}
\Cref{prop:coproduct_multiindex} and \Cref{lem:extended_coproduct} imply 
\begin{align} 
\thinspace & 
 (\Pi_\tM\tilde{\sA}_\tM\otimes\id) 
 \mathring\Delta_\tM^+(z_4^n) \label{eq:identify_multiindex} \\[0.5em]
 = & \ n! \sum_{k=2}^{n-1} \sum_{\substack{j_1,\dots,j_{n-1}\geqs0 \\j_1 + j_2 + \dots + j_{n-1} = k\\ 
 j_1 + 2j_2 + \dots + (n-1)j_{n-1} = n}} 
 \prod_{p=2}^{n-1} \frac{1}{j_p!} \biggpar{\frac{\Pi_\tM\tilde{\sA}_\tM(\sY_p)}{p!}}^{j_p} 
 \frac{1}{j_1!} z_2^{k-j_1} z_4^{j_1}\;. \notag
\end{align} 
The only difference with $(\eta\circ\hat\W)(X^n)$ 
is that the terms $k=1$ and $k=n$ are missing. The term $k=1$ allows only for $j_n = 1$  
while all other $j_i$ equal to $0$, and accounts for $-\sigma_ n z_2$.
The term $k=n$ allows only for $j_1 = n$, while all other $j_i$ equal to $0$, and accounts for 
the term $z_4^n$.
\end{proof}

\subsection{Taking care of the boundary terms}

It remains to add the boundary terms to the coproduct $\mathring\Delta_\tM$. The result is as follows.

\begin{proposition}
\label{prop:commute_H_M} 
Define a linear map $\Theta_\tM:\spanM\to\spanM$ by 
\begin{equation}
\label{eq:gamma_n} 
 \Theta_\tM(z_4^n) 
 = \gamma_n 
 := -\sigma_n z_2 - \Pi_\tM\tilde{\sA}_\tM(z_4^n)\unit_\tM
\end{equation} 
for all $n\geqs2$, while $\Theta_\tM(z^\beta) = 0$ in all other cases.
Then the following diagram commutes: 
\begin{equation}  
\label{comdiag:free_multiindex_bndry}
\begin{tikzcd}[column sep=large, row sep=large]
H
\arrow[r, "\eta"] 
\arrow[d, "\hat\W"']
& \spanM
\arrow[d, "(\Pi_\tM\tilde\sA_\tM\otimes\id)\Delta_\tM^+ + \Theta_\tM"] 
\\
\R[Y] \otimes H
\arrow[r, "\eta"] 
& \spanM
\end{tikzcd}
\end{equation}
where~$\Delta_\tM^+$ denotes the \emph{full} extended coproduct.
\end{proposition}
\begin{proof}
This follows immediately by adding the relations 
\begin{align*}
 (\Pi_\tM\tilde\sA_\tM \otimes \id)(z_4^n\otimes\id) 
 &= \Pi_\tM \tilde\sA_\tM(z_4^n) \unit_\tM\;, \\
 (\Pi_\tM\tilde\sA_\tM \otimes \id)(\id\otimes z_4^n) 
 &= z_4^n\;,
\end{align*}
to~\eqref{eq:etaWHXn}, and incorporating the extra terms into $\Theta_\tM$. 
\end{proof}

\subsection{Proof of \Cref{thm:main}}

To lighten notations, we introduce two maps $\chi_\tM:\spanM\to\spanM$ 
and $\chi_\tF:\spanF\to\spanF$ given by 
\begin{align}
 \chi_\tM &= (\Pi_\tM\tilde\sA_\tM \otimes \id) \Delta_\tM^+ + \Theta_\tM\;, \\
 \chi_\tF &= (\Pi_N\tilde\sA \otimes \id) \DeltaCK + \Theta_\tF\;, 
\end{align}
where $\Theta_\tF$ is defined by 
\begin{equation}
\label{eq:ThetaF} 
 \Theta_\tF \circ \sP_\tM = \sP_\tM \circ \Theta_\tM\;.
\end{equation} 
Then the results obtained so far show that the following diagram commutes:
\begin{equation}  
\label{comdiag:proof_main}
\begin{tikzcd}[column sep=large, row sep=large]
\R[X] 
\arrow[r, "\iota", hook] 
\arrow[rrr, bend left, "\sP"]
\arrow[d, "\W"']
& H
\arrow[r, "\eta"] 
\arrow[d, "\hat\W"']
& \spanM
\arrow[r, "\sP_\tM"] 
\arrow[d, "\chi_\tM"]
& \spanF
\arrow[d, "\chi_\tF"]
\arrow[rd, bend left, "\PiNBPHZ + \Pi_N\Theta_\tF"]
\\
\R[X,Y] 
\arrow[r, "\iota", hook] 
\arrow[rrr, bend right, "\sP"']
& \R[Y]\otimes H
\arrow[r, "\eta"] 
& \spanM
\arrow[r, "\sP_\tM"] 
& \spanF
\arrow[r, "\Pi_N"]
& \R
\end{tikzcd}
\end{equation}
Indeed, the left square is the commutative diagram~\eqref{comdiag:free}
for algebra-valued moments $\hat\mu_\sX$, 
the middle square's commutativity is shown in \Cref{prop:commute_H_M}, 
and the right square is commutative thanks to~\eqref{comdiag:M-G} and~\eqref{eq:ThetaF}. 
The remaining commutativity relations follow from the definitions. 

It follows directly from \Cref{prop:Wick_map} that 
\begin{equation*}
 \W(\e^{-\alpha X})
 = \sum_{n\geqs0} \frac{(-\alpha)^n}{n!} \W(X^n) 
 = \W(-\alpha,X) 
 = \e^{-\alpha X - K_\sX(-\alpha)}\;,
\end{equation*} 
where 
\begin{equation*}
 K_\sX(-\alpha)
 = \Lambda(\kappa_\sX)(-\alpha)
 = \sum_{n\geqs2} \frac{(-\alpha)^n}{n!} \sigma_n Y
 =: \beta Y\;.
\end{equation*} 
This determines the value of the mass renormalisation term $\beta$. 
The fact that $\W(X^n)$ is a Bell polynomial has been shown in \Cref{ssec:Bell}. 

It remains to derive an expression for the energy renormalisation term $\gamma$. 
This is determined by the fact that 
\begin{align*}
 \log\expec{\e^{-\alpha X - \beta Y}}
 &= (\Pi_N \circ \sP)(\e^{-\alpha X - \beta Y}) \\
 &= \bigpar{\bigbrak{\PiNBPHZ + \Pi_N\Theta_\tF} \circ \sP}(\e^{-\alpha X}) \\
 &= (\PiNBPHZ \circ \sP)(\e^{-\alpha X}) + \gamma\;,
\end{align*}
where we have set 
\begin{equation*}
 \gamma = \Pi_N\Theta_\tF\circ \sP(\e^{-\alpha X})\;.
\end{equation*} 
It follows that 
\begin{equation*}
 \log\expec{\e^{-\alpha X - \beta Y - \gamma}}
 = \log\expec{\e^{-\alpha X - \beta Y}} - \gamma 
 = \PiNBPHZ \circ \sP(\e^{-\alpha X})\;,
\end{equation*} 
which is convergent by \Cref{lem:BPHZ}.

The constant $\gamma$ can be made explicit by noting that~\eqref{eq:ThetaF} 
implies 
\begin{align*}
(\Theta_\tF \circ \sP_M)(z_4^n)
&= \sP_\tM(\gamma_n) \\
&= -\sigma_n \sP_\tM(z_2)
- \Pi_\tM \tilde{\sA}_\tM(z_4^n) \\
&= -\sigma_n \sP_\tM(z_2)
- \Pi_N \tilde{\sA}(\sP(X^n)) \;.
\end{align*}
Therefore, 
\begin{align*}
\gamma &= (\Pi_N\Theta_\tF\circ \sP)(\e^{-\alpha X}) \\
&= (\Pi_N\Theta_\tF\circ \sP_\tM)(\e^{-\alpha z_4}) \\
&= \sum_{n\geqs2} \frac{(-\alpha)^n}{n!} (\Pi_N\Theta_\tF\circ \sP_\tM)(z_4^n) \\
&= -\sum_{n\geqs2} \frac{(-\alpha)^n}{n!} \sigma_n \Pi_N\sP_\tM(z_2)
- \sum_{n\geqs2} \frac{(-\alpha)^n}{n!} (\Pi_N\tilde{\sA} \circ \sP)(X^n) \\
&= -\sum_{n\geqs2} \frac{(-\alpha)^n}{n!} \sigma_n \Pi_N\sP(Y) 
- (\Pi_N\tilde{\sA} \circ \sP)(\e^{-\alpha X})\;.
\end{align*}
Note however that $\Pi_N\sP(Y) = 0$ since self-pairings are not allowed. Therefore, 
$\gamma$ reduces to~\eqref{eq:gamma} as claimed. 
In the last line, to write the second sum as an exponential, we have 
used the convention $\sP(\unit) = 0$, which is natural since $\sP$ describes the logarithm
of an expectation, and the fact that $\Pi_N\sP(X) = 0$, again because self-pairings are not allowed.

\subsection{The case of subdivergences of degree smaller than $-1$}
\label{ssec:Taylor} 

As mentioned in \Cref{rem:Taylor}, if a Feynman diagram $\Gamma$ has subdivergences 
of degree below $-1$, the definition~\eqref{eq:CK_coproduct} of the Connes--Kreimer 
coproduct $\DeltaCK$ has to be modified for the BPHZ theo\-rem, \Cref{thm:BPHZ},  
to remain true. This is done by adding edge and node decorations to Feynman diagrams, 
where the node decorations represent additional monomials in the valuation, while edge
decorations represent derivatives of the Green function. Denote by $\Gamma^{\fn}_{\fe}$ the graph 
$\Gamma = (\sV,\sE)$ equipped with a node decoration $\fn: \sV\to\smash{\N_0^3}$ and an edge 
decoration $\fe: \sE\to\smash{\N_0^3}$. 
Its degree is defined by 
\begin{equation*}
 \deg(\Gamma^{\fn}_{\fe})
 = d\bigpar{|\sV|-1}
 + \sum_{v\in\sV} |\fn(v)| 
 + \sum_{e\in\sE} \bigbrak{d-2 - |\fe(e)|}\;,
\end{equation*} 
where for any $\fl\in\N_0^3$, we set $|\fl| = 
\sum_{i=1}^3 |\fl_i|$. 
Given a distinguished vertex $v^\star\in\sV$, the valuation of the
decorated graph is 
\begin{equation}
\label{eq:PiN_Taylor} 
 \Pi_N(\Gamma^{\fn}_{\fe};v^\star)
 = \int_{(\T^d)^{\sV\setminus{v^\star}}} \prod_{e\in\sE} 
\partial^{\fe(e)}G_N(x_{e_+}-x_{e_-})
\prod_{w\in\sV\setminus v^\star} (x_w - x_{v^\star})^{\fn(w)}
\6x\;,
\end{equation}
where $\partial^{\fe} = \prod_{i=1}^3 \partial^{\fe_i}$ and 
$x^{\fn} = \prod_{i=1}^3 x_i^{\fn_i}$. Since the derivative along an edge depends 
on its orientation, the graph $\Gamma$ should actually be considered as a directed graph. 

The Connes--Kreimer coproduct is then modified by adding to it terms with additional de\-co\-ra\-tions. 
We will refrain from giving the general definition here, which can be found in~\cite[(2.19)]{Hairer_BPHZ}, 
but give instead an illustrative example, when the graph is assumed to have a degree in $(-2,1]$. 
Indicating the orientation of edges only where it matters, one has for the reduced coproduct 
\begin{equation*}
 \DeltaCKred \FGIIIplusDirected
 = \hspace{-3mm}\FGIIIxIIdecoleft{}{}{}{} \otimes \hspace{-3mm}\FGIIIxIIdecoright{}{}{}{}
 - \sum_{i=1}^3 \FGIIIxIIdecoleft{e_i}{}{}{} \otimes \hspace{-3mm}\FGIIIxIIdecoright{}{}{e_i}{}
\;,
\end{equation*} 
where the $e_i$ denote the canonical basis vectors, and non-zero decorations have been placed 
near the node or edge they apply to. 

In principle, the multi-index framework could be extended to decorated graphs, by adding  
suitable information to the multi-indices. However, this is not needed in our situation. 
Indeed, the proof of \Cref{lem:subdivergences} shows that for $d<4$, all subdivergences 
have a degree larger than $-2$. This means that as in the above example, one can only 
extract graphs that have at most one node decoration of the form $e_i$. However, 
\eqref{eq:PiN_Taylor} shows that the valuation of such graphs vanishes by symmetry, since 
it involves the integral of an odd function. Therefore, decorated graphs can be ignored.